\documentclass[reqno,11pt]{amsart}
\usepackage{amsmath}
\usepackage{amssymb}
\usepackage{tabularx}
\usepackage{enumerate}
\usepackage[dvips]{psfrag,graphicx}
\usepackage{color}
\usepackage{subfigure}
\usepackage{cases}
\usepackage{bbm}

\usepackage{mathrsfs}
\usepackage{empheq}
\usepackage{amsfonts,amssymb,amsmath}
\usepackage{epsfig}

\usepackage{comment}
 
\makeatletter
\@addtoreset{equation}{section}
  
\makeatother
\newtheorem{thm}{Theorem}[section]
\newtheorem{lem}[thm]{Lemma}   
\newtheorem{cor}[thm]{Corollary}

\theoremstyle{definition}
\newtheorem{defn}{Definition}[section]
\theoremstyle{definition}

\renewcommand{\Im}{\operatorname{Im}}

\begin{document}
\title[Morse index of steady-states to the SKT model]
{Morse index of steady-states to the SKT model 
with Dirichlet boundary conditions}
 \thanks{This research was
partially supported by JSPS KAKENHI Grant Number 22K03379.}
\author[K. Kuto]{Kousuke Kuto$^\dag$}
\author[H. Sato]{Homare Sato$^\ddag$}
\thanks{$\dag$ Department of Applied Mathematics, 
Waseda University, 
3-4-1 Ohkubo, Shinjuku-ku, Tokyo 169-8555, Japan.}
\thanks{$\ddag$ Department of Pure and Applied Mathematics, 
Graduate School of Fundamental Science and Engineering,
Waseda University, 
3-4-1 Ohkubo, Shinjuku-ku, Tokyo 169-8555, Japan.}
\thanks{{\bf E-mail:} \texttt{kuto@waseda.jp}}
\date{\today}

\begin{abstract} 
This paper deals with the stability analysis for steady-states
perturbed by
the full cross-diffusion limit of the SKT
model with Dirichlet boundary conditions.
Our previous result showed that positive steady-states
consist of the branch of small coexistence type
bifurcating from the trivial solution and the branches
of segregation type bifurcating from points
on the branch of small coexistence type.
This paper shows the Morse index of steady-states
on the branches
and 
constructs the local unstable manifold around each
steady-state of which the dimension is equal to
the Morse index.
\end{abstract}

\maketitle

\section{Introduction}
In this paper, we consider the following Lotka-Volterra 
competition model with equal cross-diffusion terms:
\begin{equation}\label{para}
\begin{cases}
u_{t}=\Delta [\,(1+\alpha v)u\,]+
u(\lambda -b_{1}u-c_{1}v)
\quad&\mbox{in}\ \Omega\times (0,T),\\
v_{t}\,=\Delta [\,(1+\alpha u)v\,]+
v(\lambda -b_{2}u-c_{2}v)
\quad&\mbox{in}\ \Omega\times (0,T),\\
u=v=0
\quad&\mbox{on}\ \Omega\times (0,T),\\
u(\,\cdot\,,0)=u_{0}\ge 0,\quad
v(\,\cdot\,,0)=v_{0}\ge 0
\quad&\mbox{in}\ \Omega,
\end{cases}
\end{equation}
where $\Omega(\subset\mathbb{R}^{N})$
is a bounded domain with a smooth boundary $\partial\Omega$
if $N\ge 2$;
a bounded interval $(-\ell,\ell)$ if $N=1$.
In 1979,
for the purpose of realizing segregation phenomena of two 
competing species by reaction-diffusion equations,
Shigesada, Kawasaki and Teramoto \cite{SKT}
proposed a population model consisting of 
the Lotka-Volterra competition system with
random-, self- and cross-diffusion terms.
Since the pioneering work, a class of Lotka-Volterra
systems with cross-diffusion like \eqref{para} is called
{\it the SKT model} cerebrating the authors of \cite{SKT}.

In \eqref{para}, the unknown functions 
$u(x,t)$ and $v(x,t)$ represent the population density of
two competing species, respectively;
$\lambda$,
$b_{i}$,
$c_{i}$
$(i=1,2)$
and
$\alpha$
are positive constants,
where $\lambda$ can be interpreted as the amount of resources
for both species;
$(b_{1}, c_{2})$ and $(c_{1}, b_{2})$ are coefficients of 
intra- and inter- specific competition, respectively.
The cross-diffusion term $\alpha\Delta (uv)$
represents an inter-species repulsive interaction
of diffusion and describes a situation where
each species promotes their own diffusion more
where there are more other species.
We refer to the book by Okubo and Levin \cite{OL}
for the bio-mechanism of diffusion terms such as
the cross-diffusion.

Concerning the solvability for a class of quasilinear
parabolic systems including \eqref{para},
in a series of works \cite{Am1, Am2, 
Am3}, 
Amann established the following time-local
well-posedness in the Sobolev space:
\begin{thm}[\cite{Am1, Am2, Am3}]\label{localsolthm}
Assume $u_{0}$,
$v_{0}\in W^{1,p}_{0}(\Omega )$
with $p\in (N, \infty) $.
Then \eqref{para}
has a unique solution $(u,v)$ satisfying
$$
u,\,v\in
C([0,T_{m}); W^{1,p}_{0}(\Omega ))
\cap
C^{2,1}(\overline{\Omega}\times (0,T_{m})),
$$
where 
$T_{m}$ is a maximal existence time.
Moreover, if
$T_{m}<\infty$,
then
$$
\lim_{t\nearrow T_{m}}(
\|u(\,\cdot\,,t)\|_{W^{1,p}_{0}(\Omega)}+
\|u(\,\cdot\,,t)\|_{W^{1,p}_{0}(\Omega)}
)=\infty.
$$
\end{thm}
Concerning the time-global existence of all the solutions
obtained in Theorem \ref{localsolthm},
Kim \cite{Ki} showed $T_{m}=\infty$ in 
the one-dimensional case where $N=1$,
and in the sequel,
Lou and Winkler \cite{LW} assured $T_{m}=\infty$
when $N=2, 3$ with the additional condition
that $\Omega $ is convex. 
It should be noted that proofs in \cite{Ki} and \cite{LW}
assumed homogeneous Neumann boundary conditions,
but some modifications in their proofs assure
$T_{m}=\infty$ also under homogeneous Dirichlet 
boundary conditions as \eqref{para}.
See also \cite{Le}.

Next we introduce the bifurcation structure for
steady-states of \eqref{para} obtained by Inoue and the authors
\cite{IKS}. 
The associated stationary problem of \eqref{para}
is reduced to the following Dirichlet problem of
nonlinear elliptic equations
\begin{equation}\label{SP}
\begin{cases}
\Delta [\,(1+\alpha v)u\,]+
u(\lambda -b_{1}u-c_{1}v)=0
\ \ &\mbox{in}\ \Omega,\\
\Delta [\,(1+\alpha u)v\,]+
v(\lambda -b_{2}u-c_{2}v)=0
\ \ &\mbox{in}\ \Omega,\\
u\ge 0,\ \ v\ge 0
\ \ &\mbox{in}\ \Omega,\\
u=v=0\ \ &\mbox{on}\ \partial\Omega.
\end{cases}
\end{equation}
It is possible to verify that any weak solution
$(u,v)$ with $u\not\equiv 0$ and $v\not\equiv 0$ 
belongs to $C^{2}(\overline{\Omega})^{2}$ and
satisfies $u>0$ and $v>0$ in $\Omega$
by virtue of the elliptic regularity theory and
the maximum principle (e.g., \cite{GT}).
Throughout the paper, such a solution $(u,v)$ of \eqref{SP}
will be called a positive solution.

In order to express the bifurcation structure,
we introduce the following eigenvalue problem:
\begin{equation}\label{eg}
-\Delta\varPhi=\lambda \varPhi\quad\mbox{in}\ \Omega,
\qquad
\varPhi=0\quad\mbox{on}\ \partial\Omega.
\end{equation}
Hereafter, all eigenvalues of \eqref{eg} will be denoted by
\begin{equation}\label{lambdaj}
(0<)\lambda_{1}<\lambda_{2}\le\lambda_{3}\le\cdots\le\lambda_{j}
\le\cdots
\quad\mbox{with}\quad\lim_{j\to\infty}\lambda_{j}=\infty,
\end{equation}
counting multiplicity.
It is known from \cite{KY1, KY2, Ru} that,
if $\alpha>0$ is sufficiently large,
then $\lambda_{1}$
yields a threshold for the nonexistence/existence
of positive solutions in the sense that
\eqref{SP} has no positive solution when 
$0<\lambda\le\lambda_{1}$;
at least one positive solution when $\lambda>\lambda_{1}$.
The usual function space $W^{2,p}(\Omega )\cap W^{1,p}_{0}(\Omega )$
will be often used in the later argument.
Then we denote
$$
X_{p}:=W^{2,p}(\Omega )\cap W^{1,p}_{0}(\Omega),\quad
\boldsymbol{X}_{p}:=X_{p}\times X_{p}.
$$

In \cite{IKS}, 
the asymptotic behavior of positive solutions
of \eqref{SP} at the full cross-diffusion as $\alpha\to\infty$
was studied.
(As a similar perspective 
on the Neumann problem for the stationary SKT model, 
we refer to \cite{Ku1, LX, LW1, LW2, LN1, LN2, 
LNY1, LNY2, MSY, WWX, Wu, WX, ZW}
and references therein for the unilateral cross-diffusion limit,
and to
\cite{Ka, Ku2, Ku3} for 
the full cross-diffusion limit.)
It was shown in \cite{IKS} that,
if $\lambda>\lambda_{1}$ and 
$\lambda\neq\lambda_{j}$ for any $j\ge 2$,
then any sequence $\{(u_{n}, v_{n})\}$ of 
positive solutions of \eqref{SP} with $\alpha=\alpha_{n}\to\infty$
satisfies either of the following two scenarios:

\begin{enumerate}[(I)]
\item[(I)]
({\it small coexistence})
there exists a positive function $U\in C^{2}(\overline{\Omega})$
such that
\begin{equation}\label{unif}
\lim\limits_{n\to\infty}
(\alpha_{n}u_{n}, \alpha_{n}v_{n})= (U,U)
\quad\mbox{in}\ \boldsymbol{X}_{p}
\end{equation}
for any $p\in (N,\infty)$, 
passing to a subsequence if necessary,
and moreover, $U$ satisfies the following limiting equation:
\begin{equation}\label{LS1}
\begin{cases}
\Delta [\,(1+U)U\,]+\lambda U=0\ \ &\mbox{in}\ \Omega,\\
U=0\ \ &\mbox{on}\ \partial\Omega;
\end{cases}
\end{equation}

\item[(II)]
({\it complete segregation})
there exists a sign-changing function $w\in C^{2}(\overline{\Omega})$ 
such that
$$
\lim\limits_{n\to\infty}
(u_{n}, v_{n})= (w_{+}, w_{-})
\quad\mbox{uniformly in}\  
\overline{\Omega},
$$
passing to a subsequence if necessary,
where
$w_{+}(x):=\max\{w(x),0\}$ and
$w_{-}(x):=\max\{0,-w(x)\}$,
and moreover,
$w$ satisfies the following limiting equation
\begin{equation}\label{LS2}
\begin{cases}
\Delta w+\lambda w-b_{1}(w_{+})^{2}+c_{2}(w_{-})^{2}=0
\ \ &\mbox{in}\ \Omega,\\
w=0\ \ &\mbox{on}\ \partial\Omega.
\end{cases}
\end{equation}
\end{enumerate}
It was also shown that the set of all the positive solutions
of \eqref{LS1} form a curve
$$\widetilde{\mathcal{C}}_{\infty}=
\{(\lambda, U(\,\cdot\,,\lambda))\in (\lambda_{1},\infty)
\times X_{p}\}$$
with $\lim_{\lambda\searrow\lambda_{1}}U(\,\cdot\,, \lambda )=0$
in $X_{p}$, see Figures 1 and 2.
In the one-dimensional case where $\Omega=(-\ell, \ell )$,
it is known that
the set $\widetilde{\mathcal{S}}^{+}_{j,\infty}$
(resp.\,$\widetilde{\mathcal{S}}^{-}_{j,\infty}$)
of solutions of \eqref{LS2}
with exact $j-1$ zeros 
in $\Omega$
and $w'(-\ell )>0$
(resp. $w'(-\ell )<0$)
forms a curve
$$\widetilde{\mathcal{S}}^{+}_{j,\infty}=\{
(\lambda, w(\,\cdot\,, \lambda ))\in (\lambda_{j}, \infty)\times X_{p}\}\quad 
(\mbox{resp.}\ 
\widetilde{\mathcal{S}}^{-}_{j,\infty}=\{
(\lambda, \widetilde{w}(\,\cdot\,, \lambda ))\in (\lambda_{j}, \infty)\times X_{p}\}).$$
Here 
$\widetilde{\mathcal{S}}^{+}_{j,\infty}\cup\{(\lambda_{j},0)\}\cup
\widetilde{\mathcal{S}}^{-}_{j,\infty}$
forms a pitchfork bifurcation curve bifurcating from
the trivial solution at $\lambda=\lambda_{j}$
(e.g., \cite{HY}).

In \cite{IKS},
a subset of positive solutions of \eqref{SP}
was constructed by the perturbation of
$\widetilde{\mathcal{C}}_{\infty}$,
$\widetilde{\mathcal{S}}^{\pm}_{j, \infty}$
when $\alpha>0$ is sufficiently large.
More precisely, 
for any given large $\varLambda >0$,
there exists a large
$\overline{\alpha}=\overline{\alpha}(\varLambda)>0$
such that,
if $\alpha>\overline{\alpha}$,
then there exists a bifurcation curve
\begin{equation}\label{bif1}
\mathcal{C}_{\alpha, \varLambda}=\{
(\lambda, u(\,\cdot\,,\lambda, \alpha ), v(\,\cdot\,,\lambda, \alpha ))
\in (\lambda_{1}, \varLambda]\times \boldsymbol{X}_{p}\}
\end{equation}
with $\lim_{\lambda\searrow\lambda_{1}}
(u(\,\cdot\,,\lambda, \alpha ), v(\,\cdot\,,\lambda, \alpha ))=
(0,0)$ in $\boldsymbol{X}_{p}$ and
\begin{equation}\label{limc}
\lim_{\alpha\to\infty}
\alpha (u(\,\cdot\,,\lambda, \alpha ), v(\,\cdot\,,\lambda, \alpha ))
=(U(\,\cdot\,,\lambda), U(\,\cdot\,,\lambda))\quad\mbox{in}\
\boldsymbol{X}_{p}
\end{equation}
for each $\lambda\in (\lambda_{1},\varLambda]$.
Then it can be said that $\mathcal{C}_{\alpha, \varLambda}$
is the perturbation (with scaling)
of $\widetilde{\mathcal{C}}_{\infty}$ 
over the range $\lambda\in (\lambda_{1}, \varLambda]$.
It was also proved that
$\mathcal{C}_{\alpha, \varLambda}$ can be extended in
the direction $\lambda\to\infty$ as a connected 
subset of positive solutions of \eqref{SP}.

Furthermore, it was shown in \cite{IKS} that
the subsets of positive solutions of \eqref{SP} with $\Omega=(-\ell, \ell)$
can be constructed by perturbation of $\widetilde{\mathcal{S}}^{\pm}_{j, \infty}$
if $\alpha>0$ is sufficiently large.
To be precise,
for each $\lambda_{j}\in (\lambda_{1}, \varLambda )$,
there exists $\overline{\alpha}_{j}>0$ and
$\beta_{j,\alpha}\in (\lambda_{1}, \varLambda )$ 
such that,
if $\alpha>\overline{\alpha}_{j}$,
then bifurcation curves of positive solutions of
\eqref{SP} can be parameterized by one variable $s$ as follows
\begin{equation}\label{seg+}
\mathcal{S}^{+}_{j, \alpha, \varLambda}=
\{
(\lambda, u,v)=(
\xi_{j}(s, \alpha), u_{j}(\,\cdot\,,s,\alpha ), 
v_{j}(\,\cdot\,,s, \alpha ))\,:\,0<s\le T^{\alpha, +}_{j}\}
\end{equation}
and
\begin{equation}\label{seg-}
\mathcal{S}^{-}_{j, \alpha, \varLambda}=
\{
(\lambda, u,v)=(
\xi_{j}(s, \alpha), u_{j}(\,\cdot\,,s,\alpha ), 
v_{j}(\,\cdot\,,s, \alpha ))\,:\,-T^{\alpha, -}_{j}\le s<0\}
\end{equation}
with
$\lim_{s\to 0}
(\xi_{j}(s, \alpha), u_{j}(\,\cdot\,,s,\alpha ),
v_{j}(\,\cdot\,,s, \alpha ))=
(\beta_{j,\alpha}, u(\,\cdot\,, \beta_{j,\alpha}, \alpha ),
v(\,\cdot\,, \beta_{j,\alpha}, \alpha ))
\in \mathcal{C}_{\alpha, \varLambda}$
in $\mathbb{R}\times \boldsymbol{X}_{p}$
and
\begin{equation}\label{wconv}
\lim_{\alpha\to\infty}
(u_{j}(\,\cdot\,,s,\alpha ),
v_{j}(\,\cdot\,,s, \alpha ))
=(w_{+}, w_{-})\quad\mbox{uniformly in}\ \overline{\Omega },
\end{equation}
where 
\begin{equation}
w\in
\begin{cases}
\widetilde{\mathcal{S}}^{+}_{j, \alpha, \varLambda}
\quad\mbox{with}\ \ \lambda=\xi_{j}(s, \alpha)
\quad\mbox{if}\ \ s>0,\\
\widetilde{\mathcal{S}}^{-}_{j, \alpha, \varLambda}
\quad\mbox{with}\ \ \lambda=\xi_{j}(s, \alpha)
\quad\mbox{if}\ \ s<0.
\end{cases}
\nonumber
\end{equation}
Then it follows that
$\mathcal{S}^{-}_{j, \alpha, \varLambda}\cup
\{(\beta_{j,\alpha}, u(\,\cdot\,, \beta_{j,\alpha}, \alpha ),
v(\,\cdot\,, \beta_{j,\alpha}, \alpha ))
\}
\cup\mathcal{S}^{+}_{j, \alpha, \varLambda}
$
forms a pitchfork bifurcation curve bifurcating from
the solution 
$(\beta_{j,\alpha}, u(\,\cdot\,, \beta_{j,\alpha}, \alpha ),
v(\,\cdot\,, \beta_{j,\alpha}, \alpha ))$
on the branch $\mathcal{C}_{\alpha, \varLambda}$,
see Figures 1 and 3.

This paper focuses on the stability analysis for
steady-states of \eqref{para} in the case where
$\alpha$ is sufficiently large.
It will be shown that
each positive solution on $\mathcal{C}_{\alpha, \varLambda}$
is linearly unstable, whereas semi-trivial solutions,
such that one of $u$ and $v$ is positive and the other is zero in $\Omega$,
are linearly stable.
In the one-dimensional case,
we show that if $\lambda\in (\lambda_{j}, \lambda_{j+1})$ and
$\alpha>0$ is sufficiently large, then
the Morse index of 
$(\lambda, u(\,\cdot\,, \lambda, \alpha ), 
v(\,\cdot\,, \lambda, \alpha ))\in \mathcal{C}_{\alpha, \varLambda}$
is equal to $j$,
whereas the Morse index of
$(\lambda, u_{j}(\,\cdot\,,s,\alpha ), v_{j}(\,\cdot\,,s,\alpha ))
\in\mathcal{S}^{\pm}_{j, \alpha, \varLambda}$
is equal to $j-1$.
These results invoking the dynamical theory for a class of 
quasilinear parabolic equations including \eqref{para}
(e.g. \cite[p.312]{Yag}) enable us to construct the local unstable manifold of each positive steady-state whose dimension is
equal to the Morse index.

To summarize our results from the ecological view-point,
three bifurcation branches 
(consisting of semi-trivial solutions whose $v$ component vanishes;
semi-trivial solutions whose $u$ component vanishes;
positive solutions of small coexistence type)
bifurcate 
from the trivial solution at $\lambda=\lambda_{1}$.
Among three branches,
two branches of semi-trivial solutions are linearly stable and 
the branch $\mathcal{C}_{\alpha, \varLambda}$
of positive solutions of small coexistence type is linearly unstable. Following the branch $\mathcal{C}_{\alpha, \varLambda}$
in the direction of increasing $\lambda$, 
the dimension of
the local unstable manifold is equal to $1$ when $\lambda-\lambda_{1}>0$ is small, and then, 
the dimension of the local unstable manifold increases by $1$
passing each bifurcation point from which the branch of positive solutions
of segregation type bifurcate.
Each positive solution on the branch 
$\mathcal{S}^{\pm}_{j,\alpha, \varLambda}$ possesses the unstable dimension $j-1$
(i.e. the dimension of the local unstable manifold), which is
equal to the number of locations where the segregation occurs since
it can be interpreted that 
the competitive species are generally segregating near 
zeros of $w(x)$ by \eqref{wconv}.
In other words, the number of segregation points increases, 
the more unstable the steady-states become.
Indeed, in the final section of this paper, 
it will be shown that these purely mathematical results 
for the calculation of the Morse index are supported 
by the numerical simulations
(Figure 4).

The contents of this paper is as follows:
In Section 2, the linearly stable/unstable of 
the trivial and the semi-trivial steady-states will be shown
as the preliminary.
In Section 3, we drive the Morse index of 
positive steady-states of small coexistence type on 
the branch $\mathcal{C}_{\alpha, \varLambda}$.
In Section 4, in the one-dimensional case,
we show that the Morse index of 
each positive steady-state of segregation type on
the branches $\mathcal{S}^{\pm}_{j,\alpha, \varLambda}$
is equal to $j-1$.
In Section 5, we exhibit a numerical bifurcation
diagram with information on the Morse index by using 
the continuation software \texttt{pde2path}.

\section{Preliminary}
In what follows, we mainly discuss the linear stability/instability of
steady-states to \eqref{para}.
Corresponding to \eqref{SP},
we define the nonlinear operator
$F\,:\, \mathbb{R}_{+}\times \boldsymbol{X}_{p}\to\boldsymbol{Y}_{p}
(\,:=\,L^{p}(\Omega )\times L^{p}(\Omega ))$ by
$$
F(\lambda, u, v):=
-\Delta\biggl[
\begin{array}{l}
(1+\alpha v)u\\
(1+\alpha u)v
\end{array}
\biggr]
-\biggl[
\begin{array}{l}
u(\lambda -b_{1}u-c_{1}v)\\
v(\lambda -b_{2}u-c_{2}v)
\end{array}
\biggr].
$$
For each solution $(u^{*}, v^{*})$ of \eqref{SP},
we denote the linearized operator of $F$ around $(u^{*}, v^{*})$
by
$L(\lambda, u^{*}, v^{*} )$,
that is,
$L(\lambda, u^{*}, v^{*} ):=D_{(u,v)}F(\lambda, u^{*}, v^{*})$.

Hence it is easily verified that
$L(\lambda, u^{*}, v^{*} )
\in\mathcal{L}(\boldsymbol{X}_{p},
\boldsymbol{Y}_{p})$,
where $\mathcal{L}(\boldsymbol{X}_{p},
\boldsymbol{Y}_{p})$ represents
the set of bounded linear operators from $\boldsymbol{X}_{p}$
to $\boldsymbol{Y}_{p}$.
The linearized eigenvalue problem is formulated as
\begin{equation}\label{eg2}
L(\lambda, u^{*}, v^{*})
\biggl[
\begin{array}{l}
\phi\\
\psi
\end{array}
\biggr]
=
\mu
\biggl[
\begin{array}{l}
\phi\\
\psi
\end{array}
\biggr].
\end{equation}
Hence
$L(\lambda, u^{*}. v^{*})$ is represented by
\begin{equation}\label{linop}
\begin{split}
L(\lambda, u^{*},v^{*} )
\biggl[
\begin{array}{c}
\phi\\
\psi
\end{array}
\biggr]
=
&-\Delta\biggl(
\biggl[
\begin{array}{ll}
1+\alpha v^{*} & \alpha u^{*}\\
\alpha v^{*} & 1+\alpha u^{*}
\end{array}
\biggr]
\biggl[
\begin{array}{c}
\phi\\
\psi
\end{array}
\biggr]
\biggr)\\
&-
\biggl[
\begin{array}{ll}
\lambda -2b_{1}u^{*}-c_{1}v^{*} & -c_{1}u^{*}\\
-b_{2}v^{*} & \lambda -b_{2}u^{*}-2c_{2}v^{*}
\end{array}
\biggr]
\biggl[
\begin{array}{c}
\phi\\
\psi
\end{array}
\biggr].
\end{split}
\end{equation}
\begin{defn}
If all the eigenvalues of \eqref{eg2}
have positive real parts, then
the steady-state $(u^{*}, v^{*})$ is called
{\it linearly stable}.
If all the eigenvalues of \eqref{eg2} have
nonnegative real parts and there exist an eigenvalue
whose real part is equal to zero, then
$(u^{*}, v^{*})$ is called
{\it neutrally stable}.
If \eqref{eg2} has at least one eigenvalue
whose real part is negative,
then 
$(u^{*}, v^{*})$ is called
{\it linearly unstable}.
\end{defn}
Let $q(x)\in\mathbb{R}$ and 
$r(x)\ge(\not\equiv)\,0$ 
be 
H\"{o}lder continuous
functions in $\overline{\Omega }$.
Consider the following eigenvalue problem:
\begin{equation}\label{eg3}
\begin{cases}
-\Delta\phi+q(x)\phi =\mu r(x)\phi\quad&\mbox{in}\ \Omega,\\
\phi =0\quad&\mbox{on}\ \partial\Omega.
\end{cases}
\end{equation}
It is well-known that all the eigenvalues of \eqref{eg3}
consist of non-decreasing sequence
$$\mu_{1}<\mu_{2}\le\mu_{3}\le\cdots\quad
\mbox{with}\ \lim_{j\to\infty}\mu_{j}=\infty
\quad
\mbox{(counting multiplicity)},
$$
and moreover,
the least eigenvalue $\mu_{1}=\mu_{1}(q,r)$
can be characterized by the following variational formula:
$$
\mu_{1}(q,r)=
\inf
\left\{
\dfrac{\int_{\Omega}(|\nabla\phi|^{2}+
q(x)\phi^{2})}{
\int_{\Omega}r(x)\phi^{2}}\,:\,
\phi\in H^{1}_{0}(\Omega),\,
\int_{\Omega}r(x)\phi^{2}>0\right\}.
$$

\begin{lem}\label{monolem}
The least eigenvalue $\mu_{1}(q,r)$ 
is monotone increasing with respect to
$q$ and monotone decreasing with respect to $r$ 
in the following sense.
\begin{enumerate}[{\rm (i)}]
\item
If $q_{1}(x)\le(\not\equiv)\,q_{2}(x)$
in $\Omega$, then 
$\mu_{1}(q_{1}, r)<\mu_{1}(q_{2}, r)$.
\item
If $r_{1}(x)\le(\not\equiv)\,r_{2}(x)$
in $\Omega$, then 
$\mu_{1}(q, r_{1})>\mu_{1}(q, r_{2})$.
\end{enumerate}
\end{lem}

\begin{proof}
See e.g. \cite[Corollary 2.2]{CC2} and 
\cite[Theorem 4.2]{Ni}
for the proofs of 
the assertions (i) and (ii), respectively.
\end{proof}
In this section, we note the linearized stability/instability
of the trivial and the semi-trivial solutions
to \eqref{SP}.

\begin{lem}\label{tristlem}
If $\lambda\le\lambda_{1}$,
then \eqref{SP} admits only the trivial solution
$(u,v)=(0,0)$, and moreover, 
it is linearly stable when $\lambda<\lambda_{1}$
and
it is neutrally stable when $\lambda=\lambda_{1}$.
If $\lambda>\lambda_{1}$, then
the trivial solution is linearly unstable.
\end{lem}

\begin{proof}
See \cite[Lemma 3.1]{KY2}
for the nonexistence of nontrivial solutions of \eqref{SP}
when $\lambda\le\lambda_{1}$.

Setting $(u^{*}, v^{*})=(0,0)$ in \eqref{linop},
we see that
the linearized operator $L(\lambda, 0,0)$
is expressed as
$L(\lambda, 0,0 )=-\Delta-\lambda I$.
Therefore, the linearized eigenvalue problem
\eqref{eg2} with  $(u^{*},v^{*})=(0,0)$ is reduced to
the linear elliptic equations
$$
\begin{cases}
-\Delta\phi-\lambda \phi =\mu\phi\quad &\mbox{in}\ \Omega,\\
-\Delta\psi-\lambda \psi =\mu\psi\quad &\mbox{in}\ \Omega,\\
\phi =\psi =0\quad &\mbox{on}\ \partial\Omega.
\end{cases}
$$
by \eqref{linop}.
Hence the linearized stability/instability can be determined by
the sign of $\mu_{1}(-\lambda , 1)$.
In view of \eqref{lambdaj} and the Krein-Rutman theorem,
we obtain
$\mu_{1}(-\lambda_{1},1)=0$.
By using (i) of Lemma \ref{monolem},
we know
\begin{equation}\label{newstar}
\mu_{1}(-\lambda, 1)
\begin{cases}
>0\quad &\mbox{if}\ \lambda<\lambda_{1},\\
=0\quad &\mbox{if}\ \lambda=\lambda_{1},\\
<0\quad &\mbox{if}\ \lambda>\lambda_{1}.
\end{cases}
\end{equation}
Then all the assertions of Lemma \ref{tristlem} are verified. 
\end{proof}

Next we consider the linearized stability/instability of
semi-trivial solutions of \eqref{SP}.
By setting $v=0$ in \eqref{SP}, one can see that
the existence of 
semi-trivial solutions with $v=0$
is reduced to that of positive solutions  
of
the following diffusive logistic equation
\begin{equation}\label{logi}
\Delta u+u(\lambda -b_{1}u)=0
\quad\mbox{in}\ \Omega,
\qquad
u=0\quad\mbox{on}\ \partial\Omega.
\end{equation}
It is well-known that 
the existence and uniqueness of positive solutions
of \eqref{logi} hold true if and only if
$\lambda>\lambda_{1}$,
see e.g. \cite{CC1}.
Hereafter, the positive solution of \eqref{logi}
will be denoted by
$\theta(x;\lambda, b_{1})$.
It is also known that the set
$\{(\lambda, \theta(\,\cdot\,;\lambda, b_{1}))
\in (\lambda_{1},\infty)\times X_{p}\}$
forms a simple curve bifurcating from 
the trivial solution at $\lambda=\lambda_{1}$.
Hence \eqref{SP} has no 
semi-trivial solution if $\lambda\le\lambda_{1}$,
two semi-trivial solutions
$(u,v)=(\theta(\,\cdot\,;\lambda,b_{1}), 0)$
and
$(u,v)=(0, \theta(\,\cdot\,;\lambda,c_{2}))$
if $\lambda>\lambda_{1}$.

\begin{lem}\label{semitrilem}
Let $\lambda>\lambda_{1}$.
If
$\alpha>0$ is sufficiently large,
then
semi-trivial solutions
$(\theta(\,\cdot\,;\lambda,b_{1}), 0)$
and
$(0, \theta(\,\cdot\,;\lambda,c_{2}))$
are linearly stable.
\end{lem}

\begin{proof}
It obviously suffices to show the
linearized stability of 
$(\theta(\,\cdot\,;\lambda,b_{1}), 0)$.
To avoid complications, 
$\theta(\,\cdot\,;\lambda,b_{1})$ will be abbreviated as
$\theta_{\lambda}$.
Setting $(u,v)=(\theta_{\lambda}, 0)$ in \eqref{linop},
we see that
$$
L(\lambda, \theta_{\lambda}, 0)=
-\Delta
\biggl(
\biggl[
\begin{array}{ll}
1 & \alpha\theta_{\lambda}\\
0 & 1+\alpha\theta_{\lambda}
\end{array}
\biggr]
\biggl[
\begin{array}{c}
\phi\\
\psi
\end{array}
\biggr]
\biggr)
-
\biggl[
\begin{array}{ll}
\lambda -2b_{1}\theta_{\lambda} & -c_{1}\theta_{\lambda }\\
0 & \lambda -b_{2}\theta_{\lambda }
\end{array}
\biggr].
$$
Then the linearized eigenvalue problem \eqref{eg2} around
$(u^{*},v^{*})=(\theta_{\lambda}, 0)$ is reduced to
\begin{equation}\label{quasi}
\begin{cases}
\Delta (\phi +\alpha\theta_{\lambda }\psi )
+(\lambda -2b_{1}\theta_{\lambda })\phi
-c_{1}\theta_{\lambda}\psi+\mu\phi=0
\quad&\mbox{in}\ \Omega,\\
\Delta [(1+\alpha\theta_{\lambda})\psi]
+(\lambda -b_{2}\theta_{\lambda})\psi+\mu\psi=0
\quad&\mbox{in}\ \Omega,\\
\phi=\psi=0
\quad&\mbox{on}\ \partial\Omega.
\end{cases}
\end{equation}
By employing the change of variables 
$$
h(x):=\phi (x)-\psi(x),\qquad
k(x):=(1+\alpha\theta_{\lambda}(x))\psi(x),
$$
we transform \eqref{quasi} to
\begin{equation}\label{semi}
\begin{cases}
\Delta h+\lambda h-2b_{1}\theta_{\lambda}h-
\dfrac{(2b_{1}+c_{1}-b_{2})\theta_{\lambda}}
{1+\alpha\theta_{\lambda}}k
+\mu h=0
\quad&\mbox{in}\ \Omega,\\
\Delta k+\dfrac{\lambda -b_{2}\theta_{\lambda}}
{1+\alpha\theta_{\lambda }}k
+\dfrac{\mu}{1+\alpha\theta_{\lambda}}k=0
\quad&\mbox{in}\ \Omega,\\
h=k=0\quad&\mbox{on}\ \partial\Omega.
\end{cases}
\end{equation}
It will be shown that
the principal eigenvalue $\mu_{*}$ of 
\eqref{semi}
(or equivalently \eqref{quasi})
is positive.
In the case where $k=0$, 
\eqref{semi} is reduced to
$$
\begin{cases}
-\Delta h+(2b_{1}\theta_{\lambda}-\lambda )h=\mu h
\quad&\mbox{in}\ \Omega,\\
h=0\quad&\mbox{on}\ \partial\Omega.
\end{cases}
$$
It is possible to check that
the least eigenvalue
$\mu =\mu_{1}(2b_{1}\theta_{\lambda}-\lambda,1)$
is positive.
Indeed, it follows from \eqref{logi} that
$\theta_{\lambda}$ satisfies
$$
\begin{cases}
-\Delta\theta_{\lambda}+(b_{1}\theta_{\lambda}-\lambda )\theta_{\lambda}=0,
\quad\theta_{\lambda }>0\quad&\mbox{in}\ \Omega,\\
\theta_{\lambda}=0\quad&\mbox{on}\ \partial\Omega.
\end{cases}
$$
This fact with the Krein-Rutman theorem implies
$\mu_{1}(b_{1}\theta_{\lambda}-\lambda, 1)=0$.
Then (i) of Lemma \ref{monolem}
leads to
$$\mu_{1}(2b_{1}\theta_{\lambda}-\lambda,1)>
\mu_{1}(b_{1}\theta_{\lambda}-\lambda,1)=0.$$
Therefore, we deduce that,
if $h\neq 0$ and $k=0$ in \eqref{semi},
then $\mu>0$.

In the other case where $k\neq 0$,
we focus on the second equation of \eqref{semi}
to know that
$$
\mu\ge
\mu_{1}\biggl(
\dfrac{b_{2}\theta_{\lambda}-\lambda}{1+\alpha\theta_{\lambda}},
\dfrac{1}{1+\alpha\theta_{\lambda}}\biggr).
$$
To show the positivity of the right-hand side,
we employ the following scaling
$$
\sigma_{1}(\alpha ):=
\dfrac{1}{\alpha}
\mu_{1}\biggl(
\dfrac{b_{2}\theta_{\lambda}-\lambda}{1+\alpha\theta_{\lambda}},
\dfrac{1}{1+\alpha\theta_{\lambda}}\biggr)
=\mu_{1}
\biggl(
\dfrac{b_{2}\theta_{\lambda}-\lambda}{1+\alpha\theta_{\lambda}},
\dfrac{\alpha}{1+\alpha\theta_{\lambda}}\biggr),
$$
and
show $\sigma_{1} (\alpha )>0$ if
$\alpha >0$ is sufficiently large.
By virtue of the Dini theorem, one can see that
$$
\lim_{\alpha\to\infty}
\dfrac{b_{2}\theta_{\lambda}-\lambda}{1+\alpha\theta_{\lambda}}=0
\quad\mbox{and}\quad
\lim_{\alpha\to\infty}
\dfrac{\alpha}{1+\alpha\theta_{\lambda}}=\dfrac{1}{\theta_{\lambda}}
$$
uniformly in any compact subset $\Omega'$ of $\Omega$.
Then (ii) of Lemma \ref{monolem} implies that
$$
\liminf_{\alpha\to\infty}\sigma_{1}(\alpha )
>
\mu_{1}\biggl(0, \zeta (x)\dfrac{1}{\theta_{\lambda}}\biggr),
$$ 
where $\zeta (x)$ represents a smooth function satisfying
$0\le\zeta (x)\le 1$
for all $x\in\Omega$
and
$\zeta (x)=1$
for all $x\in\Omega'$
and
$\mbox{supp}\,\zeta
\subset\Omega$
(with
$\mbox{dist}\,(\mbox{supp}\,\zeta, \partial\Omega )>0$).
Since $\mu_{1}(0, \zeta\theta_{\lambda}^{-1})$
is the least eigenvalue of 
$$
-\Delta\phi =\mu \zeta(x)\dfrac{1}{\theta_{\lambda}}\phi
\quad\mbox{in}\ \Omega,\qquad
\phi=0\quad\mbox{on}\ \partial\Omega,
$$
then $\mu_{1}(0, \zeta\theta_{\lambda}^{-1})>0$.
Then we know that
$\sigma_{1}(\alpha)>0$ if $\alpha>0$ is sufficiently large.
That is to say,
if $k\neq 0$ in \eqref{semi} and
$\alpha>0$ is sufficiently large,
then $\mu >0$.

Consequently, we can conclude that
the principal eigenvalue $\mu_{*}$ of \eqref{semi}
(or equivalently \eqref{quasi}) is positive if $\alpha>0$
is sufficiently large.
The proof of Lemma \ref{semitrilem} is accomplished.
\end{proof}

\section{Morse index of small coexistence states}
This section is devoted to the stability analysis for
positive steady-states of small coexistence type
on the branch $\mathcal{C}_{\alpha, \varLambda}$
(expressed as \eqref{bif1}).
Our goal of this section is to obtain the Morse index
of each solution on $\mathcal{C}_{\alpha, \varLambda}$.
In the expression of $(\lambda, u(\,\cdot, \lambda, \alpha ), 
v(\,\cdot\,,\lambda, \alpha ))
\in \mathcal{C}_{\alpha, \varLambda}$,
we employ the change of variable
$$
\varepsilon:=\dfrac{1}{\alpha}
$$
to introduce
$
(u_{\varepsilon}(\lambda ), v_{\varepsilon}(\lambda )):=
(u(\,\cdot, \lambda, \varepsilon^{-1} ), 
v(\,\cdot\,,\lambda, \varepsilon^{-1} ))$
and 
$$
(U_{\varepsilon}(\lambda ), V_{\varepsilon}(\lambda )):=
\dfrac{1}{\varepsilon }(u_{\varepsilon }(\lambda ), 
v_{\varepsilon }(\lambda )).
$$
Substituting $(u^{*}, v^{*})=(u_{\varepsilon }(\lambda ), 
v_{\varepsilon }(\lambda ))$
into \eqref{linop}, we see that
\begin{equation}\label{31eigen}
\begin{split}
L(\lambda, u_{\varepsilon }(\lambda ), 
v_{\varepsilon }(\lambda ))
\biggl[
\begin{array}{c}
\phi\\
\psi
\end{array}
\biggr]
=
&-\Delta
\biggl(
\bigg[
\begin{array}{ll}
1+V_{\varepsilon }(\lambda ) & U_{\varepsilon }(\lambda )\\
V_{\varepsilon }(\lambda ) & 1+U_{\varepsilon }(\lambda )
\end{array}
\biggr]
\biggl[
\begin{array}{c}
\phi\\
\psi
\end{array}
\biggr]
\biggr)
-
\lambda 
\biggl[
\begin{array}{c}
\phi\\
\psi
\end{array}
\biggr]\\
&+\varepsilon
\biggl[
\begin{array}{ll}
2b_{1}U_{\varepsilon }(\lambda )+c_{1}V_{\varepsilon }(\lambda ) & c_{1}U_{\varepsilon }(\lambda )\\
b_{2}V_{\varepsilon }(\lambda ) & b_{2}U_{\varepsilon }(\lambda )+2c_{2}V_{\varepsilon }(\lambda )
\end{array}
\biggr]
\biggl[
\begin{array}{c}
\phi\\
\psi
\end{array}
\biggr].
\end{split}
\end{equation}
Then the linearized eigenvalue problem 
\eqref{eg2} around 
$(u^{*}, v^{*})=(u_{\varepsilon}(\lambda ), 
v_{\varepsilon }(\lambda ))$ is reduced to
the following Dirichlet problem:
\begin{equation}\label{cegn}
\begin{cases}
\Delta
[(1+V_{\varepsilon }(\lambda ))\phi +U_{\varepsilon }(\lambda )\psi]
+\lambda \phi\\
-\varepsilon
\left[
\{2b_{1}U_{\varepsilon }(\lambda )+c_{1}V_{\varepsilon }(\lambda )\}\phi
+c_{1}U_{\varepsilon }(\lambda )\psi\right]
+\mu\phi=0
\quad&\mbox{in}\ \Omega,\\
\Delta
[V_{\varepsilon }(\lambda )\phi+(1+U_{\varepsilon }(\lambda ))\psi]
+\lambda \psi\\
-\varepsilon
\left[b_{2}V_{\varepsilon }(\lambda )\phi+
\{b_{2}U_{\varepsilon }(\lambda )+2c_{2}V_{\varepsilon }(\lambda )\}\psi\right]
+\mu\psi=0
\quad&\mbox{in}\ \Omega,\\
\phi=\psi=0
\quad&\mbox{on}\ \partial\Omega.
\end{cases}
\end{equation}
In view of \eqref{limc}, we recall that
$(U_{\varepsilon }(\lambda ), V_{\varepsilon }(\lambda ))\to
(U(\lambda ),U(\lambda ))$
in $\boldsymbol{X}_{p}$ as $\varepsilon\to 0$,
where we abbreviate
$U(\lambda )=U(\,\cdot\,,\lambda)$.
Then setting $\varepsilon\to 0$
in \eqref{cegn}, we obtain the limiting
eigenvalue problem as follows:
\begin{equation}\label{limeg}
\begin{cases}
\Delta
[(1+U(\lambda))\phi +U(\lambda )\psi]
+\lambda \phi
+\mu\phi=0
\quad&\mbox{in}\ \Omega,\\
\Delta
[U(\lambda )\phi+(1+U(\lambda ))\psi]
+\lambda \psi
+\mu\psi=0
\quad&\mbox{in}\ \Omega,\\
\phi=\psi=0
\quad&\mbox{on}\ \partial\Omega.
\end{cases}
\end{equation}
To know the set of all the eigenvalues of \eqref{limeg},
we prepare the following lemma:
\begin{lem}\label{isolem}
Suppose that $r(x)\,(\not\equiv 0)$ 
is a nonnegative and H\"older continuous
function.
Let $U(\lambda )$ be the positive solution of \eqref{LS1}.
Then it holds that
$$
\mu_{1}\biggl(-\dfrac{\lambda}{1+U(\lambda )},
r\biggr)=0
\quad\mbox{for any}\ \lambda>\lambda_{1}.
$$
\end{lem}

\begin{proof}
Setting
$
Z(\lambda )=(1+U(\lambda ))U(\lambda )$
in \eqref{LS1},
one can see that
$Z(\lambda )$ satisfies
\begin{equation}\label{tuika}
-\Delta Z(\lambda )-\dfrac{\lambda }{1+U(\lambda )}Z(\lambda )=0
\quad\mbox{in}\ \Omega,\qquad
Z(\lambda )=0\quad\mbox{on}\ \partial\Omega.
\end{equation}
Thanks to the positivity that $Z(\lambda )>0$ in $\Omega$,
the application of the Krein-Rutman theorem to \eqref{tuika}
implies that
$$
\mu_{1}\biggl(
-\dfrac{\lambda}{1+U(\lambda )},q\biggr)=0
\quad\mbox{for any}\ \lambda>\lambda_{1}.
$$
The proof of Lemma \ref{isolem} is complete.
\end{proof}

\begin{lem}\label{neglem}
Suppose that $\lambda>\lambda_{1}$.
Then, all the eigenvalues of \eqref{limeg} consist of
the union of
$$
\{\lambda_{j}-\lambda\}^{\infty}_{j=1},
\qquad
\biggl\{\mu_{j}\biggl(
-\dfrac{\lambda}{1+2U(\lambda )}, \dfrac{1}{1+2U(\lambda )}\biggr)
\biggr\}^{\infty}_{j=1}.
$$
Furthermore, all the eigenvalues contained in the latter set are positive.
\end{lem}

\begin{proof}
By the change of variables
$$
h(x):=\phi(x)-\psi (x)
\quad\mbox{and}\quad
k(x):=(1+2U(\lambda ))(\phi (x)+\psi (x))
$$
we reduce \eqref{limeg}
to a pair of eigenvalue problems for single equations
with separate variables:
\begin{equation}\label{1steg}
-\Delta h-\lambda h=\mu h
\quad\mbox{in}\ \Omega,\qquad
h=0\quad\mbox{on}\ \partial\Omega
\end{equation}
and
\begin{equation}\label{2ndeg}
-\Delta k-\dfrac{\lambda }{1+2U(\lambda )}k=
\dfrac{\mu }{1+2U(\lambda )}k
\quad\mbox{in}\ \Omega,\qquad
k=0\quad\mbox{on}\ \partial\Omega.
\end{equation}
Hence the set of eigenvalues of \eqref{limeg} coincides with
the union of the sets of eigenvalues of \eqref{1steg} and \eqref{2ndeg}.

It is obvious that \eqref{1steg} admits nontrivial solutions
if and only if $\lambda+\mu=\lambda_{j}$.
This fact means that all the eigenvalues of \eqref{1steg}
are arranged as
$
\{\lambda_{j}-\lambda\}^{\infty}_{j=1}$.

Then,
by the combination of 
Lemma \ref{isolem} with
$r=1/(1+2U(\lambda ))$
and
(i) of Lemma \ref{monolem},
we deduce that
$$
\mu_{1}\biggl(
-\dfrac{\lambda}{1+2U(\lambda )},
\dfrac{1}{1+2U(\lambda )}
\biggr)
>
\mu_{1}\biggl(
-\dfrac{\lambda}{1+U(\lambda )},
\dfrac{1}{1+2U(\lambda )}
\biggr)=0
\quad\mbox{for any}\ \lambda>\lambda_{1}.
$$
This fact implies that,
if $k$ is not trivial,
then $\mu$ is necessarily positive.

Therefore, the negative eigenvalues of
\eqref{limeg} coincide with the negative elements of the set
$\{\lambda_{j}-\lambda\}^{\infty}_{j=1}$.
Hence we obtain the assertion of Lemma \ref{neglem}.
\end{proof}
Next we consider the Morse index of
any solutions on the bifurcation curve
$\mathcal{C}_{\alpha, \varLambda}$ introduced by \eqref{bif1}.
The following result asserts that
positive solutions with the Morse index $-1$ 
bifurcate from the trivial solution at $\lambda=\lambda_{1}$, 
and as $\lambda$ increases along $\mathcal{C}_{\alpha, \varLambda}$, 
the Morse index of positive solutions decreases by one 
for each time $\lambda$ exceeds the bifurcation points 
$\beta_{2,\alpha}<\beta_{3,\alpha}<\cdots<\beta_{\overline{j}, \alpha}$.
Hereafter the Morse index of 
$(\lambda, u(\,\cdot\,,\lambda, \alpha ), v(\,\cdot\,,\lambda, \alpha ))\in
\mathcal{C}_{\alpha, \varLambda}$
will be denoted by
$$i(\lambda, u(\,\cdot\,,\lambda, \alpha ), v(\,\cdot\,,\lambda, \alpha )),$$
namely,
$i(u(\,\cdot\,,\lambda, \alpha ), v(\,\cdot\,,\lambda, \alpha ))$
denotes the number of negative eigenvalues of
\eqref{eg2}
with $(u^{*}, v^{*})=(u(\,\cdot\,,\lambda, \alpha ), 
v(\,\cdot\,,\lambda, \alpha ))$.
\begin{thm}\label{mainthm1}
Let
$\overline{j}:=\max\{j\in\mathbb{N}\,:\,\lambda_{j}<\varLambda\,\}$.
Suppose that $\lambda_{1},\lambda_{2},\cdots,
\lambda_{\overline{j}}$ in \eqref{lambdaj}
are simple eigenvalues.
Then there exists a large $\overline{\alpha}=\overline{\alpha}
(\varLambda )$ such that,
if $\alpha>\overline{\alpha}$ and
$\lambda\in (\beta_{j, \alpha}, \beta_{j+1, \alpha}]\cap 
(\lambda_{1},\varLambda]$
with $\beta_{1,\alpha}:=\lambda_{1}$, then
$i(\lambda,
u(\,\cdot\,,\lambda, \alpha ), v(\,\cdot\,,\lambda, \alpha ))=j$.
\end{thm}

\begin{proof}
We study the linearized eigenvalue problem \eqref{eg2}
around $(u^{*},v^{*})=(u_{\varepsilon }(\lambda ),
v_{\varepsilon }(\lambda ))$.
Repeating the argument to get \eqref{31eigen},
we recall that \eqref{eg2} is equivalent to \eqref{cegn}.
By the change of variables
$$
h(x):=\phi(x)-\psi (x)
\quad\mbox{and}\quad
k(x):= (1+2V_{\varepsilon }(\lambda))\phi(x)+
(1+2U_{\varepsilon }(\lambda ))\psi (x),
$$
we reduce \eqref{cegn} to
\begin{equation}\label{Gdef}
\begin{split}
&\Delta
\biggl[
\begin{array}{c}
h\\
k
\end{array}
\biggr]
+
(\lambda +\mu )\biggl[
\begin{array}{ll}
1 & 0\\
\frac{U_{\varepsilon }(\lambda )-V_{\varepsilon }(\lambda )}
{1+U_{\varepsilon }(\lambda ) + V_{\varepsilon }(\lambda )}
&
\frac{1}
{1+U_{\varepsilon }(\lambda ) + V_{\varepsilon }(\lambda )}
\end{array}
\biggr]
\biggl[
\begin{array}{c}
h\\
k
\end{array}
\biggr]
\\
&-\dfrac{\varepsilon}{2(1+U_{\varepsilon }(\lambda ) + V_{\varepsilon }(\lambda ))}
\biggl[
\begin{array}{ll}
2b_{1}U_{\varepsilon }(\lambda )+
(c_{1}-b_{2})V_{\varepsilon }(\lambda )
&
(c_{1}-b_{2})U_{\varepsilon }(\lambda )
-2c_{2}V_{\varepsilon }(\lambda )\\
2b_{1}U_{\varepsilon }(\lambda )+
(c_{1}+b_{2})V_{\varepsilon }(\lambda )
&
(c_{1}+b_{2})U_{\varepsilon }(\lambda )
+2c_{2}V_{\varepsilon }(\lambda )
\end{array}
\biggr]\\
&\times
\biggl[
\begin{array}{ll}
1+2U_{\varepsilon}(\lambda ) & 1\\
-1-2V_{\varepsilon}(\lambda ) & 1
\end{array}
\biggr]
\biggl[
\begin{array}{c}
h\\
k
\end{array}
\biggr]
=
\biggl[
\begin{array}{c}
0\\
0
\end{array}
\biggr]
\end{split}
\end{equation}
with homogeneous Dirichlet boundary conditions 
$h=k=0$ on $\partial\Omega$.
For $(\lambda, h, k, \mu, \varepsilon )\in
(\lambda_{1}, \varLambda )\times\boldsymbol{X}_{p}
\times\mathbb{C}\times\mathbb{R}$,
we define $G( \lambda, h, k, \mu, \varepsilon )\in
\boldsymbol{Y}_{p}$ by the left-hand side of
\eqref{Gdef}.
Furthermore,
we define $H(\lambda, h, k, \mu, \varepsilon )\in
\boldsymbol{Y}_{p}\times\mathbb{R}$ by
$$
H(\lambda, h, k, \mu, \varepsilon ):=
\biggl[
\begin{array}{l}
G(\lambda, h, k, \mu, \varepsilon )\\
\|h\|^{2}_{2}+\|k\|^{2}_{2}-1
\end{array}
\biggr],
$$
where $\|\,\cdot\,\|_{2}$
denotes the usual norm of $L^{2}(\Omega )$.
Since $G(\lambda, h, k, \mu, \varepsilon )$
is linear with respect to $(h,k)$,
then it is reasonable to regard $\boldsymbol{X}_{p}$
and $\boldsymbol{Y}_{p}$
as
\begin{equation}\label{XpYp}
\boldsymbol{X}_{p}=
[W^{2,p}(\Omega: \mathbb{R})\cap W^{1,p}_{0}(\Omega: \mathbb{R})]
\times
[W^{2,p}(\Omega: \mathbb{C})\cap W^{1,p}_{0}(\Omega: \mathbb{C})],
\quad
\boldsymbol{Y}_{p}=L^{p}(\Omega: \mathbb{C})\times L^{p}(\Omega: \mathbb{C})
\end{equation}
in the proof of Theorem \ref{mainthm1}.

Here we recall the eigenvalue problem
\eqref{limeg} at the limit $\varepsilon\to 0$.
It follows from Lemma \ref{neglem} that
all the eigenvalues of 
\eqref{limeg}
with $\lambda=\lambda^{*}\in (\lambda_{1}, \infty)$ 
consist of
$$
\lambda_{1}-\lambda^{*}<\lambda_{2}-\lambda^{*}<\cdots
$$
and
$$
(\,0<\,)\,
\mu_{1}\biggl(
-\dfrac{\lambda^{*}}{1+2U(\lambda^{*})},
\dfrac{1}{1+2U(\lambda^{*})}\biggr)
<
\mu_{2}\biggl(
-\dfrac{\lambda^{*}}{1+2U(\lambda^{*})},
\dfrac{1}{1+2U(\lambda^{*})}\biggr)
<\cdots.
$$
In order to count the negative eigenvalues,
we set
$$
\mu_{i}^{*}:=\lambda_{i}-\lambda^{*}
\quad\mbox{for}\ i\in\mathbb{N}.
$$
In the case where
$\lambda^{*}\in (\lambda_{j}, \lambda_{j+1})$
with some $j\in\mathbb{N}$,
all the negative eigenvalues of \eqref{limeg} 
with $\lambda=\lambda^{*}$ consist of
$\{\mu^{*}_{i}\}^{j}_{i=1}$,
that is,
\begin{equation}\label{reg}
\mu^{*}_{1}<\mu^{*}_{2}<\cdots<\mu^{*}_{j}\,(\,<0<\,)\,
\mu^{*}_{j+1}<\cdots.
\end{equation}
In the other case where
$\lambda^{*}=\lambda_{j}$
with some $j\ge 2$,
all the negative eigenvalues of \eqref{limeg} 
with $\lambda=\lambda^{*}$ consist of
$\{\mu^{*}_{i}\}^{j-1}_{i=1}$,
that is,
\begin{equation}\label{cri}
\mu^{*}_{1}<\mu^{*}_{2}<\cdots<\mu^{*}_{j-1}
<\mu^{*}_{j}=0
<
\mu^{*}_{j+1}<\cdots.
\end{equation}

We 
consider \eqref{1steg}-\eqref{2ndeg}
with $(\lambda, \mu)=(\lambda^{*}, \mu_{i}^{*})$.
Then the corresponding eigenfunctions are obtained as
$$
(h^{*},k^{*})=t(\varPhi_{i},0)
\quad\mbox{with}\ t\neq 0,
$$
where
$\varPhi_{i}$ denotes an $L^{2}$ normalized eigenfunction
of $-\Delta$ with the eigenvalue $\lambda=\lambda_{i}$
under the homogeneous Dirichlet boundary condition,
namely,
$\varPhi_{i}$ satisfies
\eqref{eg} with $\lambda=\lambda_{i}$ and
$\|\varPhi_{i}\|_{2}=1$.
It follows that
\begin{equation}\label{H0}
H(\lambda^{*}, \varPhi_{i}, 0, \mu^{*}_{i}, 0)=0.
\end{equation}
Our strategy is to
construct the zero-level set of $H$
near
$(\lambda^{*}, \varPhi_{i}, 0, \mu^{*}_{i}, 0)
\in \mathbb{R}\times \boldsymbol{X}_{p}\times
\mathbb{C}\times\mathbb{R}$
by the applications of the implicit function theorem.
To do so, we need to check that
$$D_{(h,k,\mu)}H(\lambda^{*}, \varPhi_{i},
0, \mu^{*}_{i}, 0)
\in\mathcal{L}(\boldsymbol{X}_{p}\times\mathbb{C},
\boldsymbol{Y}_{p}\times\mathbb{R})$$
is an isomorphism.
In view of the left-hand side of \eqref{Gdef},
one can verify that
\begin{equation}
\begin{split}
&D_{(h,k,\mu)}H
(\lambda^{*}, \varPhi_{i}, 0, \mu^{*}_{i}, 0)
\left[
\begin{array}{c}
h\\
k\\
\mu
\end{array}
\right]
=
\left[
\begin{array}{l}
D_{(h,k,\mu)}G
(\lambda^{*}, \varPhi_{i}, 0, \mu^{*}_{i}, 0)
\left[
\begin{array}{c}
h\\
k\\
\mu
\end{array}
\right]
\\
2\int_{\Omega}\varPhi_{i}h
\end{array}
\right]\\
&=
\left[
\begin{array}{l}
\Delta\biggl[
\begin{array}{l}
h\\
k
\end{array}
\biggr]
+(\lambda^{*}+\mu^{*}_{i})
\biggl[
\begin{array}{ll}
1 & 0\\
0 & \frac{1}{1+2U(\lambda^{*})}
\end{array}
\biggr]
\biggl[
\begin{array}{l}
h\\
k
\end{array}
\biggr]
+
\mu
\biggl[
\begin{array}{ll}
1 & 0\\
0 & \frac{1}{1+2U(\lambda^{*})}
\end{array}
\biggr]
\biggl[
\begin{array}{l}
\varPhi_{i}\\
0
\end{array}
\biggr]
\\
2\int_{\Omega}\varPhi_{i}h
\end{array}
\right].
\end{split}
\nonumber
\end{equation}
By virtue of $\lambda^{*}+\mu^{*}_{i}=\lambda_{i}$,
we see that
\begin{equation}
D_{(h,k,\mu)}H
(\lambda^{*}, \varPhi_{i}, 0, \mu^{*}_{i}, 0)
\left[
\begin{array}{c}
h\\
k\\
\mu
\end{array}
\right]
=
\left[
\begin{array}{l}
\Delta h+\lambda_{i}h+\mu\varPhi_{i}\\
\Delta k+\frac{\lambda_{i}}{1+2U(\lambda^{*})}k\\
2\int_{\Omega}\varPhi_{i}h
\end{array}
\right].
\nonumber
\end{equation}
In order to verify that
$D_{(h,k,\mu)}H(\lambda^{*}, \varPhi_{i}, 0,
\mu^{*}_{i}, 0)
\in\mathcal{L}(\boldsymbol{X}_{p}\times\mathbb{C},
\boldsymbol{Y}_{p}\times\mathbb{R})$
is an isomorphism,
we suppose that
$D_{(h.k.\mu)}
H(\lambda^{*}, \varPhi_{i}, 0,
\mu^{*}_{i}, 0)(h,k,\mu)^{T}=0$,
that is,
\begin{equation}\label{isosys}
\begin{cases}
-\Delta h-\lambda_{i}h=\mu\varPhi_{i}
\quad&\mbox{in}\ \Omega,\\
-\Delta k-\frac{\lambda_{i}}{1+2U(\lambda^{*})}k=0
\quad&\mbox{in}\ \Omega,\\
\int_{\Omega}\varPhi_{i}h=0,\\
h=k=0\quad&\mbox{on}\ \partial\Omega.
\end{cases}
\end{equation}
Taking the inner product of the first equation with $\varPhi_{i}$, 
we obtain
$\mu=0$.
This fact leads to 
$$
-\Delta h=\lambda_{i}h\quad\mbox{in}\ \Omega,
\qquad h=0\quad\mbox{on}\ \partial\Omega,\qquad
\displaystyle\int_{\Omega}
\varPhi_{i}h=0.
$$
Owing to the assumption that $\lambda_{i}$ is a simple
eigenvalue of \eqref{eg},
the Fredholm alternative theorem yields $h=0$.
It follows from 
Lemma \ref{isolem} and (i) of Lemma \ref{monolem} that
\begin{equation}\label{step}
0=\mu_{1}\biggl(
-\dfrac{\lambda^{*}}{1+U(\lambda^{*})},1\biggr)
<
\mu_{1}\biggl(
-\dfrac{\lambda^{*}}{1+2U(\lambda^{*})},1\biggr).
\end{equation}
Whether in the case $\lambda^{*}\in (\lambda_{j}, \lambda_{j+1})$
with some $j\in\mathbb{N}$ or the case 
$\lambda^{*}=\lambda_{j}$ with some $j\ge 2$,
it follows that $\lambda_{i}\le\lambda^{*}$
for any $i=1,2,\ldots,j$.
Then (i) of Lemma \ref{monolem} with \eqref{step} implies
$$
0
<
\mu_{1}\biggl(
-\dfrac{\lambda^{*}}{1+2U(\lambda^{*})},1\biggr)
\le
\mu_{1}\biggl(
-\dfrac{\lambda_{i}}{1+2U(\lambda^{*})},1\biggr)
\quad
\mbox{for}\ i=1,2,\ldots,j.
$$
Therefore, we know from the second equation of 
\eqref{isosys} that
$$
k=\biggl(-\Delta-\dfrac{\lambda_{i}}{1+2U(\lambda^{*})}
\biggr)^{-1}0=0
\quad
\mbox{for}\ i=1,2,\ldots,j.
$$
Consequently, we deduce that
$D_{(h,k,\mu)}H(\lambda^{*}, \varPhi_{i}, 0,
\mu^{*}_{i}, 0)
\in\mathcal{L}(\boldsymbol{X}_{p}\times\mathbb{C},
\boldsymbol{Y}_{p}\times\mathbb{C})$
is an isomorphism for each
$i=1,2,\ldots,j$.
With \eqref{H0}, the implicit function theorem
gives a neighborhood $\mathcal{U}^{*}_{i}$
of $(\lambda^{*}, \varPhi_{i}, 0, \mu^{*}_{i}, 0)
\in \mathbb{R}\times \boldsymbol{X}_{p}\times
\mathbb{C}\times\mathbb{R}$,
a small number
\begin{equation}\label{deltai}
\delta_{i}(\lambda^{*})>0
\end{equation} and
the mapping
$$
B((\lambda^{*}, 0); \delta_{i}(\lambda^{*}))\ni
(\lambda, \varepsilon )\mapsto
(h_{i}(\lambda, \varepsilon ),
k_{i}(\lambda, \varepsilon ),
\mu_{i} (\lambda, \varepsilon ))
\in\boldsymbol{X}_{p}\times\mathbb{C}
$$
of class $C^{1}$,
where
$B((\lambda^{*}, 0); \delta_{i}(\lambda^{*}))=\{
(\lambda, \varepsilon)\in\mathbb{R}^{2}\,:\,
(\lambda-\lambda^{*})^{2}+\varepsilon^{2}<
\delta_{i}(\lambda^{*})^{2}\}$,
such that
all the solutions of 
$H(\lambda, h, k, \mu, \varepsilon )=0$
in $\mathcal{U}^{*}_{i}$ are given by
$$
(h,k,\mu)=(h_{i }(\lambda, \varepsilon ),
k_{i }(\lambda, \varepsilon ),
\mu_{i } (\lambda, \varepsilon ))
\quad\mbox{for}\ 
(\lambda, \varepsilon )\in B((\lambda^{*}, 0); \delta_{i}(\lambda^{*})).
$$
Hence it holds that
\begin{equation}\label{zero}
(h_{i }(\lambda, 0 ),
k_{i }(\lambda, 0 ),
\mu_{i } (\lambda, 0 ))
=
(\varPhi_{i}, 0, \lambda_{i}-\lambda ).
\end{equation}

It will be shown that 
$\{\mu_{i}(\lambda, \varepsilon )\}^{j}_{i =1}$
are real for any
$(\lambda, \varepsilon )\in B((\lambda^{*}, 0); \delta^{*}_{i })$.
Suppose for contradiction that
$\Im\mu_{i}(\hat{\lambda }, \hat{\varepsilon })\neq 0$
with some $i\in\{1,2,\ldots, j\}$ and
$(\hat{\lambda }, \hat{\varepsilon })\in B((\lambda^{*}, 0); \delta^{*}_{i })$.
Taking the complex conjugate of \eqref{Gdef}, one can verify that
the conjugate 
$\overline{\mu}_{i}(\hat{\lambda }, \hat{\varepsilon })
\,(\,\neq \mu_{i}(\hat{\lambda }, \hat{\varepsilon })\,)$
is also an eigenvalue of \eqref{Gdef} 
(also \eqref{eg2}).
It follows that
$$
H(\hat{\lambda}, \overline{h}_{i}(\hat{\lambda}, \hat{\varepsilon}),
\overline{k}_{i}(\hat{\lambda }, \hat{\varepsilon }), \overline{\mu}_{i}(\hat{\lambda }, \hat{\varepsilon }),
\hat{\varepsilon })=0
\quad
\mbox{in addition to}
\quad
H(\hat{\lambda}, h_{i}(\hat{\lambda}, \hat{\varepsilon}),
k_{i}(\hat{\lambda }, \hat{\varepsilon }), 
\mu_{i}(\hat{\lambda }, \hat{\varepsilon }),
\hat{\varepsilon })=0.
$$
Hence this contradicts the fact that
all the solutions of $H(\lambda, h, k, \mu, \varepsilon )=0$
in $\mathcal{U}^{*}_{i}$
are uniquely determined by
$(\lambda, h_{i}(\lambda , \varepsilon ), 
k_{i}(\lambda, \varepsilon ),
\mu_{i}(\lambda, \varepsilon ), \varepsilon )$
for all $(\lambda, \varepsilon )\in
B((\lambda^{*}, 0); \delta^{*}_{i })$.

We aim to count the negative eigenvalues 
of \eqref{Gdef} (also \eqref{eg2}) when 
$\lambda$ is close to the bifurcation point
$\beta_{j}(\varepsilon ):=\beta_{j, \varepsilon }$ with
$\alpha =1/\varepsilon $.
By virtue of $\beta_{j}(\varepsilon )\to\lambda_{j}$
as $\varepsilon\to 0$,
we first set $\lambda^{*}=\lambda_{j}$ with
some $j\ge 2$.
In view of \eqref{cri} and \eqref{zero},
we see that
if 
$$
0<\varepsilon <\delta (\lambda_{j}):=\min_{i=1,\ldots, j}
\delta_{i}(\lambda_{j}),
$$
then
there are at least $j-1$ 
negative eigenvalues of \eqref{Gdef} (also \eqref{eg2}) as
$$
\mu_{1}(\lambda_{j}, \varepsilon )<\mu_{2}(\lambda_{j}, \varepsilon )<\cdots<
\mu_{j-1}(\lambda_{j}, \varepsilon )\,(\,<0\,)
$$
and
$\mu_{j}(\lambda_{j}, \varepsilon )$ lies in a neighborhood of zero.
Here we recall that
the bifurcation point
$$(\beta_{j,\alpha}, u(\,\cdot\,,\beta_{j,\alpha}, \alpha ),
v(\,\cdot\,,\beta_{j,\alpha}, \alpha ))\in\mathcal{C}_{\alpha, \varLambda}$$
satisfies
$\beta_{j,\alpha}\to\lambda_{j}$ as $\alpha\to\infty$.
Then the continuity of $\mu_{i}(\lambda, \varepsilon )$ implies that
\begin{equation}\label{bj00}
\mu_{1}(\beta_{j}(\varepsilon ), \varepsilon )<\mu_{2}(\beta_{j}(\varepsilon ), \varepsilon )<\cdots<
\mu_{j-1}(\beta_{j}(\varepsilon ), \varepsilon )\,(\,<0\,),
\end{equation}
where $\beta_{j}(\varepsilon ):=\beta_{j,\alpha}$ with $\alpha=1/\varepsilon $,
if $\varepsilon>0$ is sufficiently small.
Here we remark that \eqref{Gdef} (also \eqref{eg2}) with 
$\lambda=\beta_{j}(\varepsilon )$ necessarily
has the zero eigenvalue since 
$(\beta_{j,\alpha}, u(\,\cdot\,,\beta_{j,\alpha}, \alpha ),
v(\,\cdot\,,\beta_{j,\alpha}, \alpha ))\in\mathcal{C}_{\alpha, \varLambda}$
is a bifurcation point.
Together with the facts that $\mu_{j}(\lambda_{j}, \varepsilon )$ is the unique
eigenvalue near zero 
(for \eqref{Gdef} with $\lambda=\lambda_{j}$) 
and that $\beta_{j}(\varepsilon )\to \lambda_{j}$ as $\varepsilon\to 0$,
we deduce that
\begin{equation}\label{bj0}
\mu_{j}(\beta_{j}(\varepsilon ), \varepsilon )=0
\end{equation}
if $\varepsilon>0$ is sufficiently small.

Next it will be shown that
$\partial_{\lambda}\mu_{j}(\beta_{j}(\varepsilon ), \varepsilon )<0$.
Differentiating the first component of \eqref{Gdef} 
by $\lambda$
and setting $(\lambda, \varepsilon)=(\lambda_{j}, 0)$ 
in the resulting expression, 
one can see that
$
h_{\lambda }(\lambda, \varepsilon ):=
\partial_{\lambda } h(\lambda, \varepsilon )$
satisfies
\begin{equation}\label{der}
\Delta h_{\lambda}(\lambda_{j}, 0)+(1+
\partial_{\lambda}\mu_{j}(\lambda_{j}, 0))\varPhi_{j}+
(\lambda_{j}+\mu_{j}(\lambda_{j}, 0))h_{\lambda}(\lambda_{j}, 0)=0
\quad\mbox{in}\ \Omega.
\end{equation}
Here we note
$\mu_{j}(\lambda_{j}, 0)=0$
by setting $\varepsilon = 0$ in \eqref{bj0}.
Taking the inner product of \eqref{der} with $\varPhi_{j}$, we obtain
$\partial_{\lambda}\mu_{j}(\lambda_{j}, 0)=-1$.
We recall that $\mu_{j}(\lambda, \varepsilon )$ 
is of class $C^{1}$ near $(\lambda, \varepsilon )=(\lambda_{j}, 0)$.
Therefore, there exists a small positive number
$\delta^{\#}_{j}$ such that,
if $0<\varepsilon<\delta^{\#}_{j}$, then
\begin{equation}\label{bifu}
\partial_{\lambda}\mu_{j}(\beta_{j}(\varepsilon ), \varepsilon )<-\dfrac{1}{2}.
\end{equation}
By virtue of \eqref{bj00}, \eqref{bj0} and \eqref{bifu},
there exists
$\tau_{j}>0$
such that,
if $0<\varepsilon<\delta^{\#}_{j}$,
then
\begin{equation}\label{j-1}
\mu_{1}(\lambda, \varepsilon )<\mu_{2}(\lambda, \varepsilon )<\cdots<
\mu_{j-1}(\lambda, \varepsilon )<0
\quad\mbox{for}\ 
\lambda\in(\beta_{j}(\varepsilon )-\tau_{j}, \beta_{j}(\varepsilon )+\tau_{j})
\end{equation}
and
\begin{equation}\label{pass}
\mu_{j}(\lambda, \varepsilon )
\begin{cases}
>0\quad&\mbox{for}\ \lambda\in 
(\beta_{j}(\varepsilon )-\tau_{j}, \beta_{j}(\varepsilon )),\\
=0\quad&\mbox{for}\ \lambda=\beta_{j}(\varepsilon ),\\
<0\quad&\mbox{for}\ \lambda\in (\beta_{j}(\varepsilon ),
\beta_{j}(\varepsilon )+\tau_{j}).
\end{cases}
\end{equation}

The final task of the proof 
is to count the negative eigenvalues of \eqref{Gdef}
(also \eqref{eg2})
when $\lambda$ is away from the bifurcation point
$\beta_{j}(\varepsilon )$.
To do so, we assume
$$\lambda^{*}\in 
\biggl(
\lambda_{j}+\dfrac{\tau_{j}}{2}, \lambda_{j+1}-\dfrac{\tau_{j+1}}{2}\biggr)
=:I_{j}
$$
with some $j\in\mathbb{N}$.
By virtue of \eqref{reg} and \eqref{zero},
we see that
if 
$$
0<\varepsilon <\delta (\lambda^{*}),
$$
then
there exist negative eigenvalues of \eqref{Gdef} (also \eqref{eg2}) as
$$
\mu_{1}(\lambda^{*}, \varepsilon )<\mu_{2}(\lambda^{*}, \varepsilon )<\cdots<
\mu_{j}(\lambda^{*}, \varepsilon )\,(\,<0\,).
$$
Together with the continuity of $\mu_{i}(\lambda, \varepsilon)$
$(i=1,2,\ldots , j)$,
we find a small $\tau^{*}>0$, which is depending on $\lambda^{*}$ 
such that
if $0<\varepsilon <\delta^{*}$, then
\begin{equation}\label{j}
\mu_{1}(\lambda, \varepsilon )<\mu_{2}(\lambda, \varepsilon )<\cdots<
\mu_{j}(\lambda, \varepsilon )\,(\,<0\,)
\quad\mbox{for any}\ \lambda\in (\lambda^{*}-\tau^{*}, \lambda^{*}+\tau^{*})=:J(\lambda^{*}).
\end{equation}
By virtue of $\beta_{j}(\varepsilon )\to \lambda_{j}$ as $\varepsilon \to 0$,
we can see that
$$
\bigcup^{\overline{j}}_{j=1}\bigcup_{\lambda^{*}\in I_{j}}J(\lambda^{*})\cup\,
(\beta_{2}(\varepsilon )-\tau_{2}, \beta_{2}(\varepsilon )+\tau_{2})
\cup\cdots\cup
(\beta_{\overline{j}}(\varepsilon )-\tau_{j}, \beta_{\overline{j}}(\varepsilon )+\tau_{j})
$$
gives an open covering of the compact set
$[\lambda_{1}+\tau_{1}, \varLambda]$.
It follows that there exist
$$\lambda^{*}_{(1)},\ldots,
\lambda^{*}_{(m)}\in\bigcup^{\overline{j}}_{j=1}I_{j}$$ 
such that
$$
\bigcup^{m}_{n=1}J(\lambda^{*}_{(n)})\cup\,
(\beta_{2}(\varepsilon )-\tau_{2}, \beta_{2}(\varepsilon )+\tau_{2})
\cup\cdots\cup
(\beta_{\overline{j}}(\varepsilon )-\tau_{j}, \beta_{\overline{j}}(\varepsilon )+\tau_{j})
$$
covers $[\lambda_{1}+\tau_{1}, \varLambda]$.
Therefore, we can deduce from \eqref{j-1}, \eqref{pass} and \eqref{j} that,
if 
$$
0<\varepsilon<\min_{j=1,\ldots,\overline{j}}\delta^{\#}_{j}
\quad\mbox{and}\quad
0<\varepsilon<\min_{n=1,\ldots,m}\delta (\lambda^{*}_{(n)}),
$$
then,
for each $j=1,\ldots,\overline{j}$,
$$
\mu_{1}(\lambda, \varepsilon )<\mu_{2}(\lambda, \varepsilon )<\cdots<
\mu_{j}(\lambda, \varepsilon )<0
\quad\mbox{for any}\ 
\lambda\in(\beta_{j}(\varepsilon ), \beta_{j+1}(\varepsilon )] \cap
(\lambda_{1}+\tau_{1}, \varLambda].
$$
By a slight modification of the argument to get 
\eqref{pass}, we can show that
$$
\mu_{1}(\lambda, \varepsilon )<0
\quad\mbox{for}\ \lambda\in (\lambda_{1}, \lambda_{1}+\tau_{1}]
$$
if $\varepsilon>0$ is small enough.
Consequently, the above argument enables us to conclude that
$$
i(\lambda, u(\,\cdot\,,\lambda, \varepsilon^{-1}), 
v(\,\cdot\,,\lambda, \varepsilon^{-1}))=j
\quad\mbox{for any}\ 
\lambda\in (\beta_{j}(\varepsilon ), \beta_{j+1}(\varepsilon )]\cap
(\lambda_{1}, \varLambda],
$$
where $\beta_{1}(\varepsilon )=\lambda_{1}$.
Then the proof of Theorem \ref{mainthm1} is complete.
\end{proof}
In the one-dimensional case,
it can be proved that the Morse index
$i(\lambda, u(\,\cdot\,,\lambda, \alpha )
v(\,\cdot\,,\lambda, \alpha ))$
obtained in Theorem \ref{mainthm1}
is equal to the dimension of
the local unstable manifold of
the steady-state
$(u(\,\cdot\,,\lambda, \alpha ),
v(\,\cdot\,,\lambda, \alpha ))$ of \eqref{para}.
\begin{cor}\label{cor1}
Assume 
$\Omega=(-\ell, \ell)$. 
If
$\alpha>0$ is large,
then for each
$\lambda\in (\beta_{j, \alpha}, \beta_{j+1, \alpha})\cap
(\lambda_{1}, \varLambda]$,
there exists a local unstable manifold
$\mathcal{W}^{u}((u(\,\cdot\,,\lambda, \alpha ),
v(\,\cdot\,,\lambda, \alpha )); \mathcal{O})$
in $\boldsymbol{Z}_{\theta}:=W^{1+\theta,2}_{0}(\Omega )\times W^{1+\theta, 2}_{0}(\Omega )$
satisfying 
$\dim \mathcal{W}^{u}((u(\,\cdot\,,\lambda, \alpha ),
v(\,\cdot\,,\lambda, \alpha )); \mathcal{O})=j$,
where $\theta\in (0,1/2)$ and $\mathcal{O}
\subset
\boldsymbol{Z}_{\theta}
$ is a neighborhood of 
$(u(\,\cdot\,,\lambda, \alpha ),
v(\,\cdot\,,\lambda, \alpha ))$.
\end{cor}

\begin{proof}
In the one-dimensional case where $\Omega=(-\ell, \ell)$,
it is known that \eqref{para} admits a unique time-global
solution in the class
$
(u, v)\in C([0,\infty): \boldsymbol{Z}_{\theta})\cap
C^{\infty}((0,\infty)\times\overline{\Omega})^{2}$,
see e.g.,
\cite{Ki, Yag}.
Furthermore, it is also known
from \cite[p.520]{Yag} that
the dynamical system from the abstract 
evolution equation associated with \eqref{para}
can be determined in the universal space
$\boldsymbol{Z}_{\theta}$.
Then the result in \cite[Section 6.5]{Yag} ensures
the desired local unstable manifold.
\end{proof}

\section{Morse index of segregation states}
In this section,
the Morse index of almost all solutions on 
$\mathcal{S}^{\pm}_{j, \alpha, \varLambda}$
(in \eqref{seg+} and \eqref{seg-})
will be obtained.
Before stating the result, it should be noted that 
$\mathcal{S}^{\pm}_{j, \alpha, \varLambda}$
can be 
parameterized by 
$\lambda$ 
outside the neighborhood of bifurcation point.

\begin{lem}[\cite{IKS}]
Assume 
$\Omega=(-\ell, \ell)$. 
For any small $\eta>0$
and $j\ge 2$,
there exist small 
$\underline{\delta}^{\pm}_{j}(\eta, \varLambda )>0$
such that
for each $\alpha>1/\delta^{\pm}_{j}$,
there exists a pair of simple curves 
$\mathcal{S}^{\pm}_{j,\alpha}(\lambda_{j}+\eta, \varLambda)
\,(\,\subset 
\mathcal{S}^{\pm}_{j, \alpha, \varLambda}\,)$
consisting of positive solutions to \eqref{SP} parameterized by
$\lambda\in (\lambda_{j}+\eta, \varLambda)$ such as
$$
\mathcal{S}^{\pm}_{j,\alpha}(\lambda_{j}+\eta, \varLambda)=
\{\,(\lambda, \widetilde{u}^{\pm}_{j}(\lambda,\varepsilon ), 
\widetilde{v}^{\pm}_{j}(\lambda,\varepsilon))
\in (\lambda_{j}+\eta, \varLambda)\times\boldsymbol{X}_{p}\,\},
$$
where $\varepsilon=1/\alpha$ and
$$
(\lambda, \widetilde{u}^{\pm}_{j}(\lambda,\varepsilon ), 
\widetilde{v}^{\pm}_{j}(\lambda,\varepsilon))=
(\xi_{j}(s, \alpha ), u_{j}(\,\cdot\,,s, \alpha ),
v_{j}(\,\cdot\,,s, \alpha ))$$
with the expressions in \eqref{seg+} and \eqref{seg-}.
\end{lem}
The next result gives the Morse index 
$
i(\lambda,
\widetilde{u}^{\pm}_{j}(\lambda, \varepsilon ),
\widetilde{v}^{\pm}_{j}(\lambda, \varepsilon ))$
of any 
$(\lambda,
\widetilde{u}^{\pm}_{j}(\lambda, \varepsilon ),
\widetilde{v}^{\pm}_{j}(\lambda, \varepsilon ))
\in\mathcal{S}^{\pm}_{j, \alpha}(\lambda_{j}+\eta, \varLambda )$:
\begin{thm}\label{mainthm2}
Assume 
$\Omega=(-\ell, \ell)$.
For any 
$(\lambda, \widetilde{u}^{\pm}_{j}(\lambda, \varepsilon ),
\widetilde{v}^{\pm}_{j}(\lambda, \varepsilon ))\in
\mathcal{S}^{\pm}_{j, \alpha}(\lambda_{j}+\eta, \varLambda )$,
there exists a small positive   
$\delta$ which is independent of $\lambda
\in (\lambda_{j}+\eta,\varLambda )$ and $j\in
\{2,\ldots,\overline{j}\}$
such that,
if $0<\varepsilon<\delta$, then
$
i(\lambda,
\widetilde{u}^{\pm}_{j}(\lambda, \varepsilon ),
\widetilde{v}^{\pm}_{j}(\lambda, \varepsilon ))=j$.
\end{thm}

\begin{proof}
Our task is to count the negative eigenvalues of
\begin{equation}\label{ege}
L(\lambda, 
\widetilde{u}^{\pm}_{j}(\lambda, \varepsilon ), 
\widetilde{v}^{\pm}_{j}(\lambda, \varepsilon ))
\biggl[
\begin{array}{l}
\phi\\
\psi
\end{array}
\biggr]
=
\mu
\biggl[
\begin{array}{l}
\phi\\
\psi
\end{array}
\biggr]
\qquad\mbox{for any}\
\lambda\in [\lambda_{j}+\eta, \varLambda ].
\end{equation}
Here we fix $\lambda^{*}\in (\lambda_{j}-\eta/2,
\varLambda+\eta/2)$ arbitrarily.
In what follows, 
we abbreviate
$$
(u^{*}_{\varepsilon }, v^{*}_{\varepsilon }):=
(\widetilde{u}^{\pm}_{j}(\lambda^{*}, \varepsilon ), 
\widetilde{v}^{\pm}_{j}(\lambda^{*}, \varepsilon ))
$$
to avoid complications with subscripts.
By the change of variables
$$
h(x):=\phi (x)-\psi (x),\quad
k(x):=(1+2\varepsilon^{-1}v^{*}_{\varepsilon }(x))\phi (x)
+(1+2\varepsilon^{-1}u^{*}_{\varepsilon }(x))\psi (x),
$$
we reduce \eqref{ege} to the equation
\begin{equation}\label{Gtil}
\widetilde{G}(\lambda, h,k,\mu, \varepsilon )=0,
\end{equation}
where
the mapping 
$$
(\lambda_{j}+\eta, \varLambda )\times
\boldsymbol{X}_{p}\times\mathbb{C}\times\mathbb{R}
\ni (\lambda, h,k,\mu, \varepsilon )\mapsto
\widetilde{G}(\lambda, h,k,\mu, \varepsilon )
\in\boldsymbol{Y}_{p}$$
is defined by
\begin{equation}
\begin{split}
&\widetilde{G}(\lambda, h,k,\mu, \varepsilon ):=\\
&\dfrac{d^{2}}{dx^{2}}
\biggl[
\begin{array}{c}
h\\
k
\end{array}
\biggr]
+
(\lambda +\mu )\biggl[
\begin{array}{ll}
1 & 0\\
\frac{u^{*}_{\varepsilon }-v^{*}_{\varepsilon }}
{\varepsilon+u^{*}_{\varepsilon }+ v^{*}_{\varepsilon }}
&
\frac{\varepsilon }
{\varepsilon +u^{*}_{\varepsilon }+ v^{*}_{\varepsilon }}
\end{array}
\biggr]
\biggl[
\begin{array}{c}
h\\
k
\end{array}
\biggr]
\\
&-\dfrac{1}{2(\varepsilon+
u^{*}_{\varepsilon }+ v^{*}_{\varepsilon })}
\biggl[
\begin{array}{ll}
2b_{1}u^{*}_{\varepsilon }+
(c_{1}-b_{2})v^{*}_{\varepsilon }
&
(c_{1}-b_{2})u^{*}_{\varepsilon }
-2c_{2}v^{*}_{\varepsilon }\\
2b_{1}u^{*}_{\varepsilon }+
(c_{1}+b_{2})v^{*}_{\varepsilon }
&
(c_{1}+b_{2})u^{*}_{\varepsilon }
+2c_{2}v^{*}_{\varepsilon }
\end{array}
\biggr]
\biggl[
\begin{array}{ll}
\varepsilon +2u^{*}_{\varepsilon} & \varepsilon\\
-\varepsilon-2v^{*}_{\varepsilon} & \varepsilon
\end{array}
\biggr]
\biggl[
\begin{array}{c}
h\\
k
\end{array}
\biggr].
\end{split}
\nonumber
\end{equation}
We recall from \eqref{unif} that
\begin{equation}\label{uvlim}
\lim_{\varepsilon\to 0}(u^{*}_{\varepsilon }, v^{*}_{\varepsilon })
=(w^{*}_{+},w^{*}_{-})
\quad\mbox{uniformly in}\ [-\ell,\ell ],\\
\end{equation}
where $w^{*}\in \widetilde{\mathcal{S}}^{+}_{j,\infty}\cup
\widetilde{\mathcal{S}}^{-}_{j,\infty}$
with $\lambda=\lambda^{*}$.
Here we consider the limiting equation
$\widetilde{G}(\lambda^{*}, h,k,\mu, 0)=0$,
that is,
\begin{equation}\label{lin}
\begin{cases}
h''+\lambda^{*} h-2(b_{1}w^{*}_{+}+c_{2}w^{*}_{-})h+\mu h=0
\quad&\mbox{in}\ (-\ell, \ell),\\
k''+(\lambda^{*}+\mu)\,(\mbox{sign}\,w^{*})h-
2(b_{1}w^{*}_{+}-c_{2}w^{*}_{-})h=0\quad&\mbox{in}\ 
(-\ell, \ell),\\
h(\pm\ell )=k(\pm\ell )=0,
\end{cases}
\end{equation}
where $\mbox{sign}\,w^{*}:=w^{*}/|w^{*}|$.
We note that the first equation of \eqref{lin}
is corresponding to the linearized eigenvalue problem of
\eqref{LS2} around $w^{*}$, that is,
\begin{equation}\label{eg1dim}
-h''-f_{w}(\lambda^{*}, w^{*})h=\mu h\quad\mbox{in}\ 
(-\ell, \ell),\qquad
h(\pm\ell )=0,
\end{equation}
where
$$
f(\lambda, w)=
\lambda w-b_{1}(w_{+})^{2}+c_{2}(w_{-})^{2}=
\begin{cases}
w(\lambda -b_{1}w)\quad&\mbox{for}\ w\ge 0,\\
w(\lambda +c_{2}w)\quad&\mbox{for}\ w<0.
\end{cases} 
$$
Let $\mu_{i,0}\in\mathbb{R}$
be the $i$-th eigenvalue of \eqref{eg1dim}.
By the Sturm-Liouville theory,
it is well-known that
$$
\sharp\{
\mbox{negative eigenvalues of }\eqref{eg1dim}\}=
\sharp\{
z\in (-\ell, \ell)\,:\,w(z)=0\}
=j.
$$
Then it follows that
all the eigenvalues of \eqref{eg1dim} are arranged as
\begin{equation}\label{jko}
\mu^{*}_{1,0}<\mu^{*}_{2,0}<\cdots<\mu^{*}_{j,0}
\,(\,<0<\,)\,
\mu^{*}_{j+1,0}<\cdots.
\end{equation}
For
\eqref{eg1dim} with $\mu=\mu^{*}_{i,0}$,
let $h^{*}_{i,0}(x)$ be the eigenfunction with
$\|h^{*}_{i,0}\|_{2}=1$ and $(h^{*}_{i,0})'(-\ell )>0$.
In view of \eqref{lin}, we see that
\begin{equation}\label{Gtil0}
\widetilde{G}
(\lambda^{*}, h^{*}_{i,0}, k^{*}_{i,0}, \mu^{*}_{i,0}, 0)=0,
\end{equation}
where
$$
k^{*}_{i,0}:=
\mathcal{G}\,
[\,\{\,2(c_{2}w^{*}_{-}-b_{1}w^{*}_{+})
+(\lambda^{*}+\mu^{*}_{i,0})\,\mbox{sign}\,w^{*}\,\}
h^{*}_{i,0}\,].
$$
Here $\mathcal{G}$ represents the inverse
operator of $-d^{2}/dx^{2}$ with homogeneous
Dirichlet boundary conditions at $x=\pm\ell$.

In \eqref{Gtil},
we set $h=0$ and consider the equation
$\widetilde{G}(\lambda, 0,k,\mu, \varepsilon )=0$,
which consists of 
\begin{equation}\label{egh0}
\begin{cases}
\{(2b_{1}+c_{1}-b_{2})u^{*}_{\varepsilon }+
(c_{1}-b_{2}-2c_{2})v^{*}_{\varepsilon }\}k=0
\quad&\mbox{in}\ (-\ell, \ell),\\
- k''
-\dfrac{\varepsilon}{2(\varepsilon +u^{*}_{\varepsilon }
+v^{*}_{\varepsilon })}
\left[
2(\lambda +\mu )
-
\left\{(2b_{1}+c_{1}+b_{2})u^{*}_{\varepsilon }
+(c_{1}+b_{2}+2c_{2})v^{*}_{\varepsilon }
\right\}
\right]k
=0
\quad&\mbox{in}\ (-\ell, \ell),\\
k(\pm\ell )=0.
\end{cases}
\end{equation}
It can be shown that 
the corresponding eigenvalues
$\mu_{\varepsilon_{n}}$ satisfy
$\mu_{\varepsilon_n}\to\infty$.
Suppose for contradiction that
$\mu_{\varepsilon_n}<M$
with some positive constant $M$ independent of $n$.
With the aid of \eqref{uvlim},
we know that the operator
$$
-\dfrac{d^{2}}{dx^{2}}
-\dfrac{\varepsilon_n}{2(\varepsilon_n +u_{\varepsilon_n }
+v_{\varepsilon_n })}
\left[
2(\lambda +\mu_{\varepsilon_{n}} )
-
\left\{(2b_{1}+c_{1}+b_{2})u_{\varepsilon_n }
+(c_{1}+b_{2}+2c_{2})v_{\varepsilon_n }
\right\}
\right]
\in\mathcal{L}(X_{p}, Y_{p})
$$
is invertible if $n$ is sufficiently large.
Then in this situation,
the second equation implies
$k_{n}=0$.
However, this
is a contradiction because $\mu_{\varepsilon_{n}}$
are eigenvalues.
Therefore, we can deduce that,
if some $k\neq 0$ satisfies
$\widetilde{G}(\lambda, 0,k,\mu,\varepsilon )=0$, 
then
$\mu$ is positive and
away from zero
when $\varepsilon >0$ is sufficiently small.
Therefore, we may assume $h\neq 0$ in \eqref{Gtil}
when counting negative eigenvalues of \eqref{ege}.

Here we define the mapping
$\widetilde{H}\,:\,
(\lambda_{j}+\eta, \varLambda )\times
\boldsymbol{X}_{p}\times\mathbb{C}\times\mathbb{R}
\to\boldsymbol{Y}_{p}\times\mathbb{R}$
by
$$
\widetilde{H}(\lambda, h,k,\mu,\varepsilon ):=
\left[
\begin{array}{l}
\widetilde{G}(\lambda, h,k,\mu,\varepsilon )\\
\|h\|^{2}_{2}-1
\end{array}
\right],
$$
where $\boldsymbol{X}_{p}$ and $\boldsymbol{Y}_{p}$ 
are regarded as \eqref{XpYp}.
Then it follows from \eqref{Gtil0} that
$$
\widetilde{H}
(\lambda^{*}, h^{*}_{i,0}, k^{*}_{i,0}, \mu^{*}_{i,0}, 0)=0.
$$
In order to construct the zero-level set of 
$\widetilde{H}$ in a neighborhood of 
$$(\lambda^{*}, h^{*}_{i,0}, k^{*}_{i,0}, \mu^{*}_{i,0}, 0)\in
\mathbb{R}\times\boldsymbol{X}_{p}\times\mathbb{C}\times\mathbb{R}$$
with the aid of the implicit function theorem,
we need to show that
$$
D_{(h,k,\mu )}
\widetilde{H}
(\lambda^{*}, h^{*}_{i,0}, k^{*}_{i,0}, \mu^{*}_{i,0}, 0)
\in\mathcal{L}(\boldsymbol{X}_{p}\times\mathbb{C},
\boldsymbol{Y}_{p}\times\mathbb{R})
$$
is an isomorphism.
For this end,
we assume that
$
D_{(h,k,\mu )}
\widetilde{H}
(\lambda^{*}, h^{*}_{i,0}, k^{*}_{i,0}, \mu^{*}_{i,0}, 0)
(h,k,\mu)^{T}=0$,
namely,
\begin{equation}\label{lin2}
\begin{cases}
h''+(\lambda^{*} +\mu^{*}_{i,0})h-
2(b_{1}w^{*}_{+}+c_{2}w^{*}_{-})h+\mu h^{*}_{i,0}=0
\quad&\mbox{in}\ (-\ell, \ell),\\
k''+(\lambda^{*}+\mu^{*}_{i,0})
(\mbox{sign}\,w^{*})h
-2(b_{1}w^{*}_{+}-c_{2}w^{*}_{-})h
+\mu (\mbox{sign}\,w^{*}) h^{*}_{i,0}=0\quad&\mbox{in}\ (-\ell, \ell ),\\
\int^{\ell}_{-\ell}h^{*}_{i,0}\,h=0,\\
h(\pm\ell )=k(\pm\ell )=0.
\end{cases}
\end{equation}
By taking the inner product of the first equation of \eqref{lin2}
with $h^{*}_{i,0}$, we use the first equation of \eqref{lin} to
see $\mu =0$.
Then the first equation is reduced to
$$
h''+(\lambda^{*}+\mu^{*}_{i,0})h-
2(b_{1}w^{*}_{+}+c_{2}w^{*}_{-})h=0
\quad\mbox{in}\ (-\ell, \ell).
$$
Hence it follows that
$h=th^{*}_{i,0}$.
By the integral condition of \eqref{lin2}, 
we see $t=0$, and therefore, $h=0$.
Finally we obtain $k=0$ by the second equation of \eqref{lin2}.
Consequently, we deduce that
$
D_{(h,k,\mu )}
\widetilde{H}
(\lambda^{*}, h^{*}_{i,0}, k^{*}_{i,0}, \mu^{*}_{i,0}, 0)
$ is an isomorphism from
$\boldsymbol{X}_{p}\times\mathbb{C}$
to
$\boldsymbol{Y}_{p}\times\mathbb{R}$.
Therefore,
for each $i$,
the implicit function theorem yields a
neighborhood 
$\mathcal{U}^{*}_{i,0}$
of $(\lambda^{*},
h^{*}_{i,0}, k^{*}_{i,0}, \mu^{*}_{i,0}, 0)
\in\mathbb{R}\times\boldsymbol{X}_{p}\times\mathbb{C}\times\mathbb{R}$
and a small positive number $\delta_{i}$
and a mapping
$$
B((\lambda^{*}, 0); \delta^{*}_{i})
\ni(\lambda, \varepsilon )\mapsto
(h^{*}_{i}(\lambda, \varepsilon ), 
k^{*}_{i}(\lambda, \varepsilon ), 
\mu^{*}_{i}(\lambda, \varepsilon ))
\in\boldsymbol{X}_{p}\times\mathbb{C}
$$
of class $C^{1}$ such that
all the solutions of 
$\widetilde{H}(\lambda, h,k,\mu,\varepsilon )=0$
in $\mathcal{U}^{*}_{i,0}$ are represented by
$$
(h,k, \mu)=
(h^{*}_{i}(\lambda, \varepsilon ), 
k^{*}_{i}(\lambda, \varepsilon ), 
\mu^{*}_{i}(\lambda, \varepsilon ))\quad
\mbox{for}\ 
(\lambda, \varepsilon )\in 
B((\lambda^{*}, 0); \delta^{*}_{i}).
$$
Hence it holds that
$
(h_{i}(\lambda, 0 ), 
k_{i}(\lambda, 0 ), 
\mu_{i}(\lambda, 0 ))=
(h_{i,0}, k_{i,0}, \mu_{i,0})$,
where $
(h_{i,0}, k_{i,0}, \mu_{i,0})$
is the solution of \eqref{lin}
with $\lambda^{*}$ replaced by $\lambda$
satisfying
$\|h_{i,0}\|_{2}=1$ and $h_{i,0}'(-\ell )>0$.
Since $\mu_{i,0}\in\mathbb{R}$, then
we may assume $\mu_{i}(\lambda, \varepsilon )\in\mathbb{R}$
for $(\lambda, \varepsilon )\in B((\lambda^{*}, 0); \delta^{*}_{i})$
by 
the same argument just after \eqref{zero} 
in the proof of Theorem \ref{mainthm1}.
By the continuity of $\mu_{i}(\lambda, \varepsilon )$
with respect to $(\lambda, \varepsilon )$,
we can deduce from \eqref{jko} that
\begin{equation}\label{egorder}
\mu_{1}(\lambda, \varepsilon )<\mu_{2}(\lambda, \varepsilon )
<\cdots<\mu_{j}(\lambda, \varepsilon )
\,(\,<0<\,)\,
\mu_{j+1}(\lambda, \varepsilon )
\end{equation}
for any
$(\lambda, \varepsilon )\in 
B((\lambda^{*}, 0); \delta^{*}_{j})$.
Since
$$
\bigcup\,\biggl\{(\lambda^{*}-\delta^{*}_{i}, \lambda^{*}+\delta^{*}_{i})\,:\,
\lambda^{*}\in \biggl(\lambda_{j}-\dfrac{\eta}{2},
\varLambda+\dfrac{\eta}{2}\biggr)\biggr\}
$$
gives an open covering of the compact set
$[\lambda_{j}+\eta, \varLambda ]$,
then the similar procedure to the argument below
\eqref{j} ensures a small positive
$\delta$
which is independent of $\lambda\in [\lambda_{j}+\eta,
\varLambda ]$
and
$j\in \{2,\ldots, \overline{j}\}$
such that,
if $0<\varepsilon <\delta$,
then \eqref{egorder} holds true for
any 
$\lambda\in [\lambda_{j}+\eta,
\varLambda]$.
Then the proof of Theorem \ref{mainthm2} is complete.
\end{proof}

The Morse index 
$i(\lambda, \widetilde{u}^{\pm}_{j}(\lambda, \varepsilon ),
\widetilde{v}^{\pm}(\lambda, \varepsilon ))$
obtained in Theorem \ref{mainthm2} 
corresponds to the dimension of 
the local unstable manifolds for 
the steady-states
$(\widetilde{u}^{\pm}_{j}(\lambda, \varepsilon ),
\widetilde{v}^{\pm}(\lambda, \varepsilon ))$
of \eqref{para} with $\Omega =(-\ell, \ell)$.

\begin{cor}
In the case where $\Omega =(-\ell, \ell)$,
let $(\lambda, \widetilde{u}^{\pm}_{j}(\lambda, \varepsilon ),
\widetilde{v}^{\pm}_{j}(\lambda, \varepsilon ))
\in\mathcal{S}^{\pm}_{j,\alpha }(\lambda+\eta, \varLambda )$.
If $\alpha=1/\varepsilon $ is sufficiently large,
then there exist unstable local manifolds
$\mathcal{W}^{u}((\widetilde{u}^{\pm}_{j}(\lambda, \varepsilon ),
\widetilde{v}^{\pm}_{j}(\lambda, \varepsilon )); \mathcal{O}^{\pm}_{j})$
in $\boldsymbol{Z}_{\theta}:=W^{1+\theta,2}_{0}(\Omega )\times W^{1+\theta, 2}_{0}(\Omega )$
satisfying $\dim \mathcal{W}^{u}((\widetilde{u}^{\pm}_{j}(\lambda, \varepsilon ),
\widetilde{v}^{\pm}_{j}(\lambda, \varepsilon )); \mathcal{O}^{\pm}_{j})=j$,
where $\theta\in (0,1/2)$ and $\mathcal{O}^{\pm}_{j}\subset
\boldsymbol{Z}_{\theta}
$ are neighborhoods of 
$(\widetilde{u}^{\pm}_{j}(\lambda, \varepsilon ),
\widetilde{v}^{\pm}_{j}(\lambda, \varepsilon ))$, respectively.
\end{cor}

\begin{proof}
The proof is essentially same as that of Corollary \ref{cor1}.
\end{proof}

\begin{figure}
\begin{center}
{\includegraphics*[scale=.5]{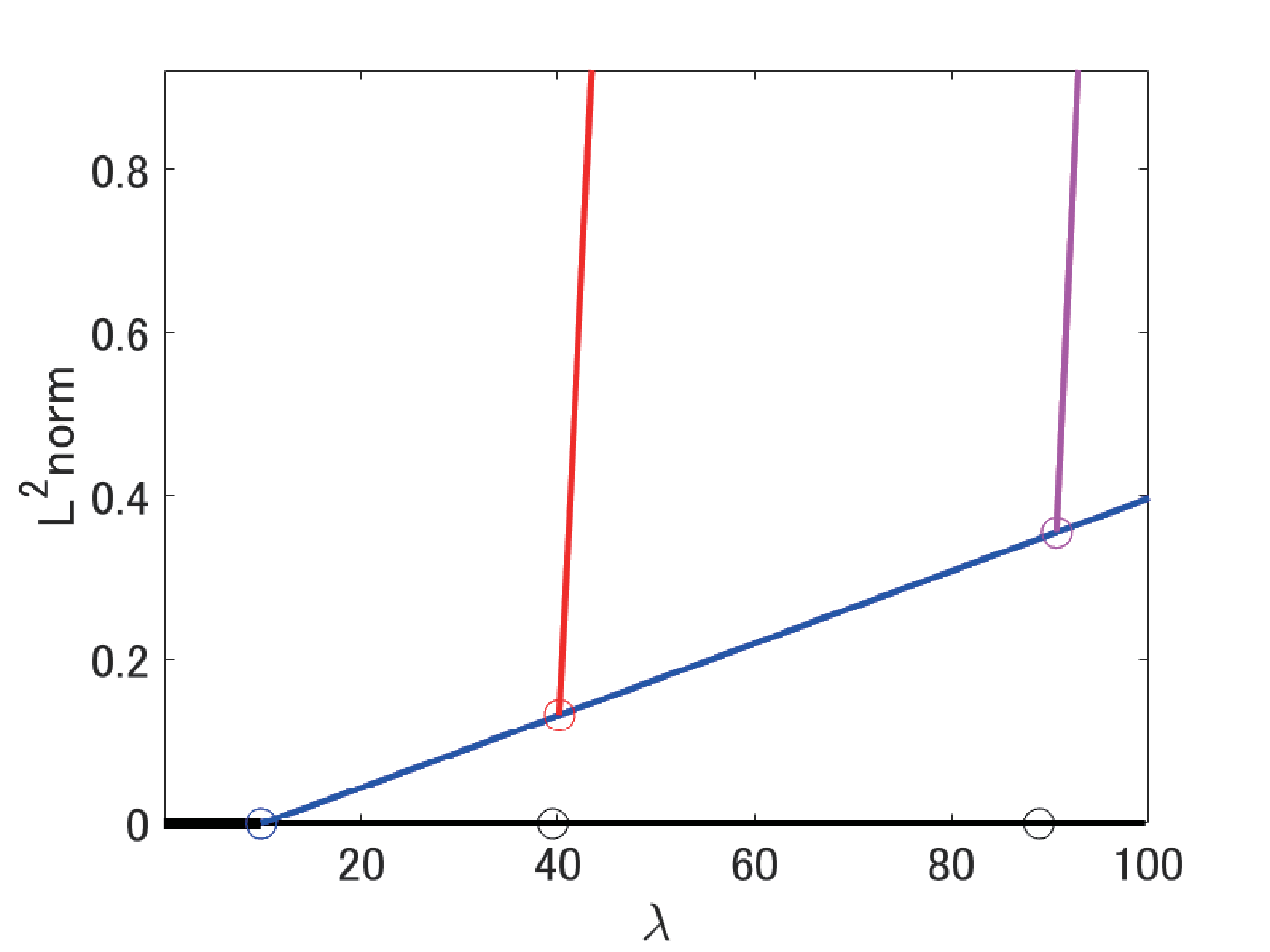}}\\
\caption{Bifurcation diagram of solutions of \eqref{SP}.}

{\includegraphics*[scale=.4]{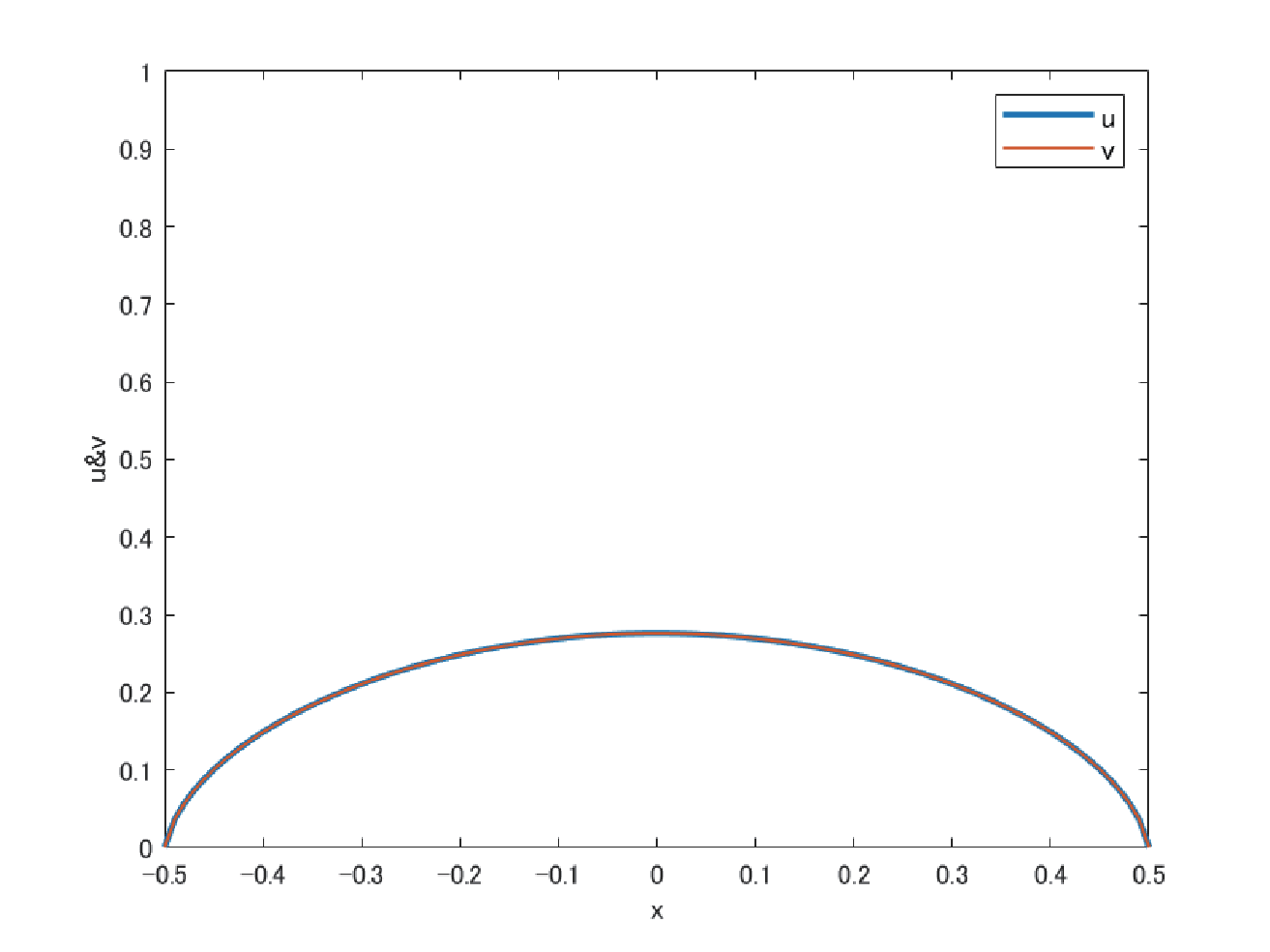}}\\
\caption{Profile of a solution on blue curve at $\lambda=59.8286$.}
\end{center}\end{figure}

\begin{figure}
\centering
\subfigure[Profile of a solution on red upper branch in Figure 1 
at $\lambda=40.3421$.]{
\includegraphics*[scale=.3]{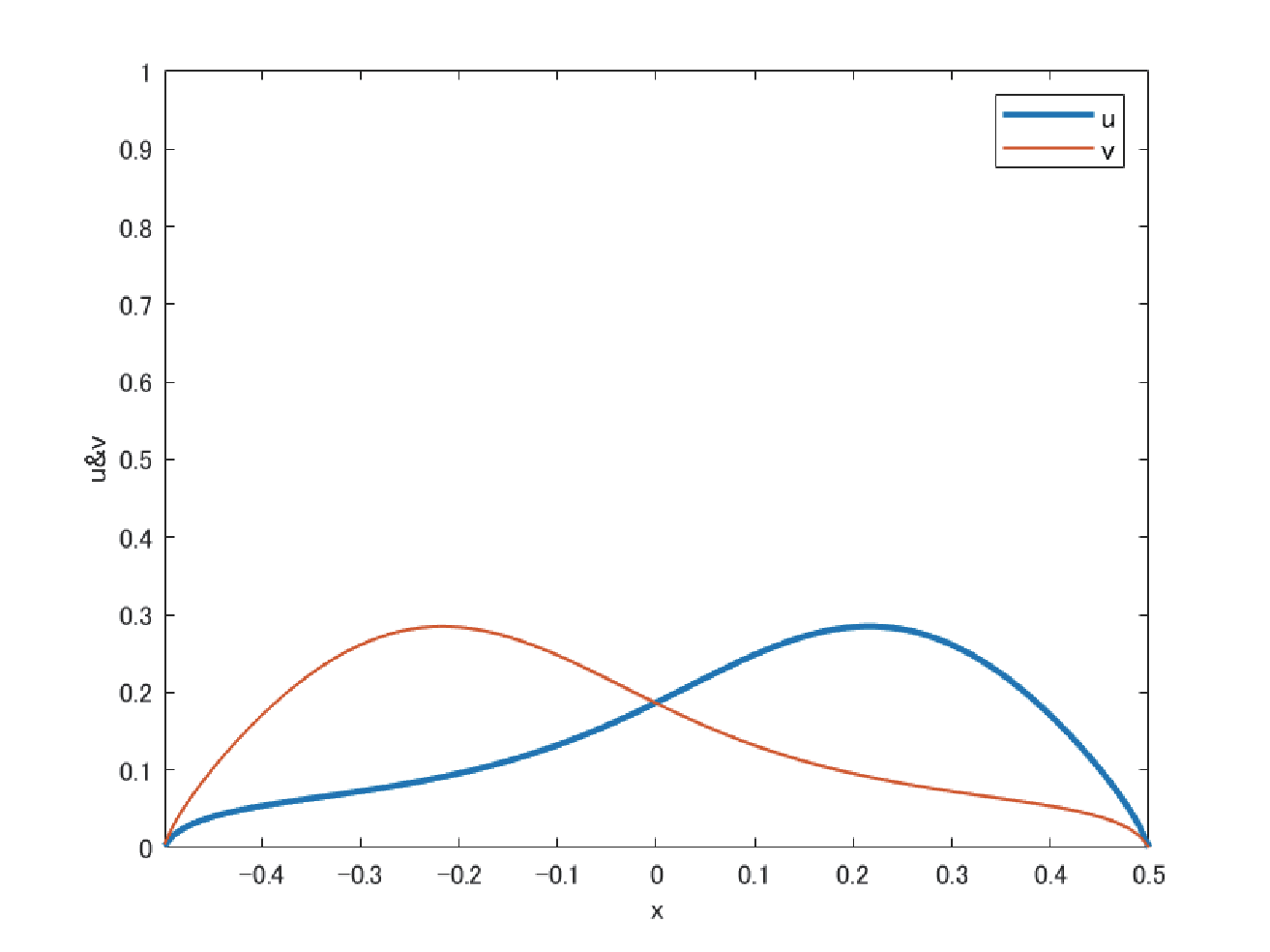}
\label{figa}}
\subfigure[Profile of a solution on red lower branch in Figure 1 
at $\lambda=40.3421$.]{
\includegraphics*[scale=.3]{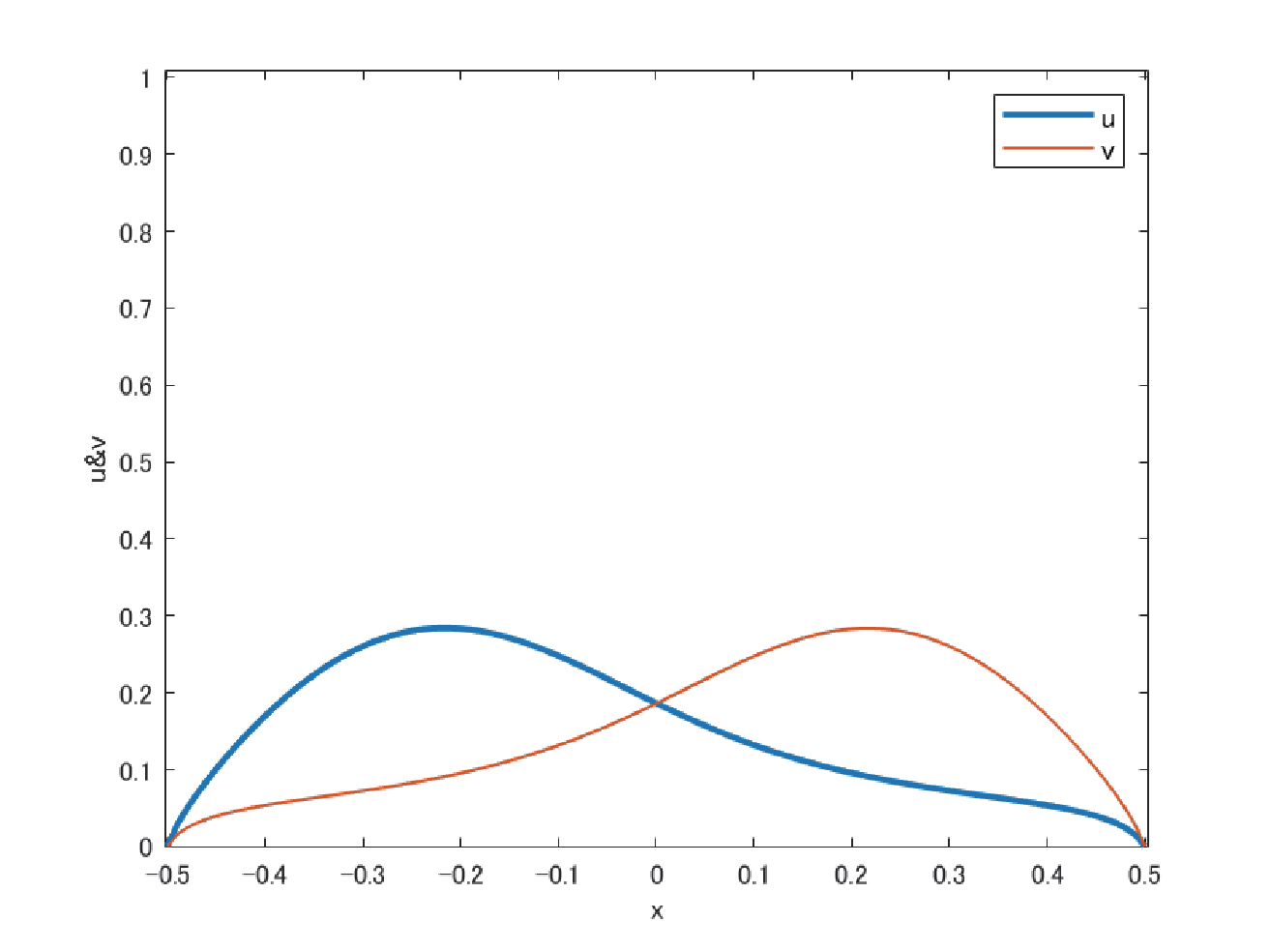}
\label{figb}
}
\centering
\subfigure[Profile of a solution on red upper branch in Figure 1 
at $\lambda=43.0673$.]{
\includegraphics*[scale=.3]{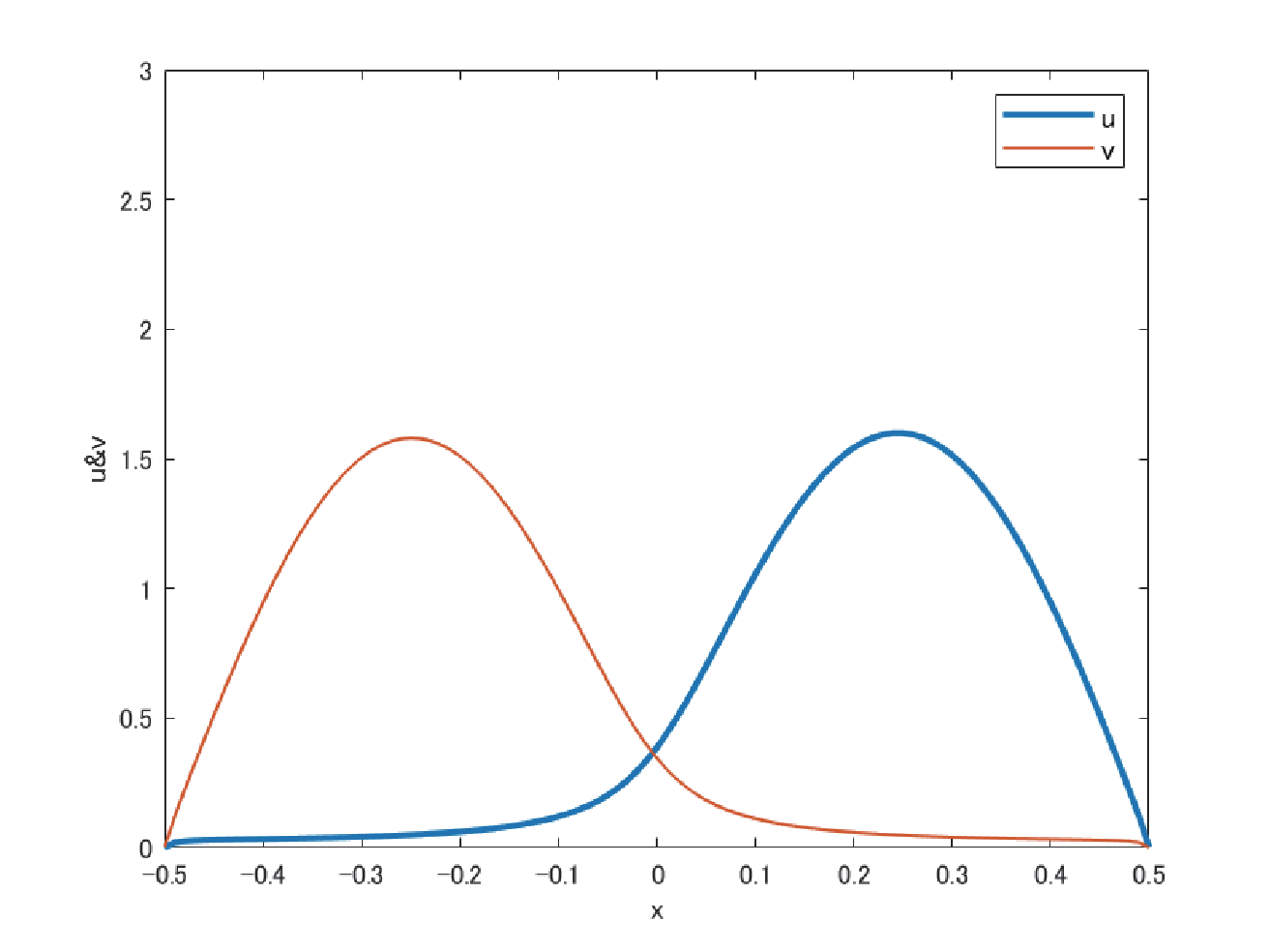}
\label{figc}}
\subfigure[Profile of a solution on red lower branch in Figure 1 
at $\lambda=43.0673$.]{
\includegraphics*[scale=.3]{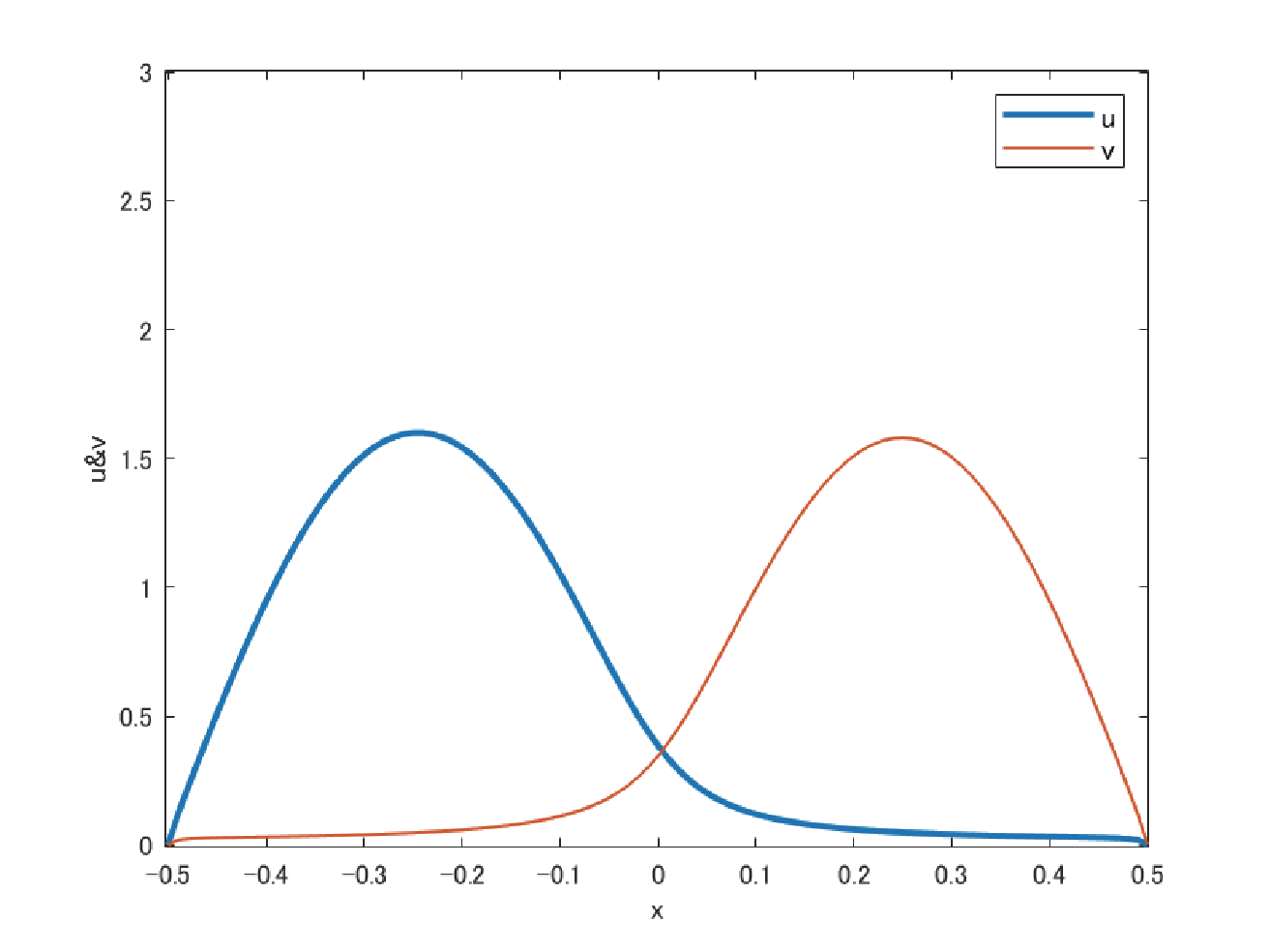}
\label{figd}
}
\centering
\subfigure[Profile of a solution on purple upper branch in Figure 1 
at $\lambda=91.5836$.]{
\includegraphics*[scale=.3]{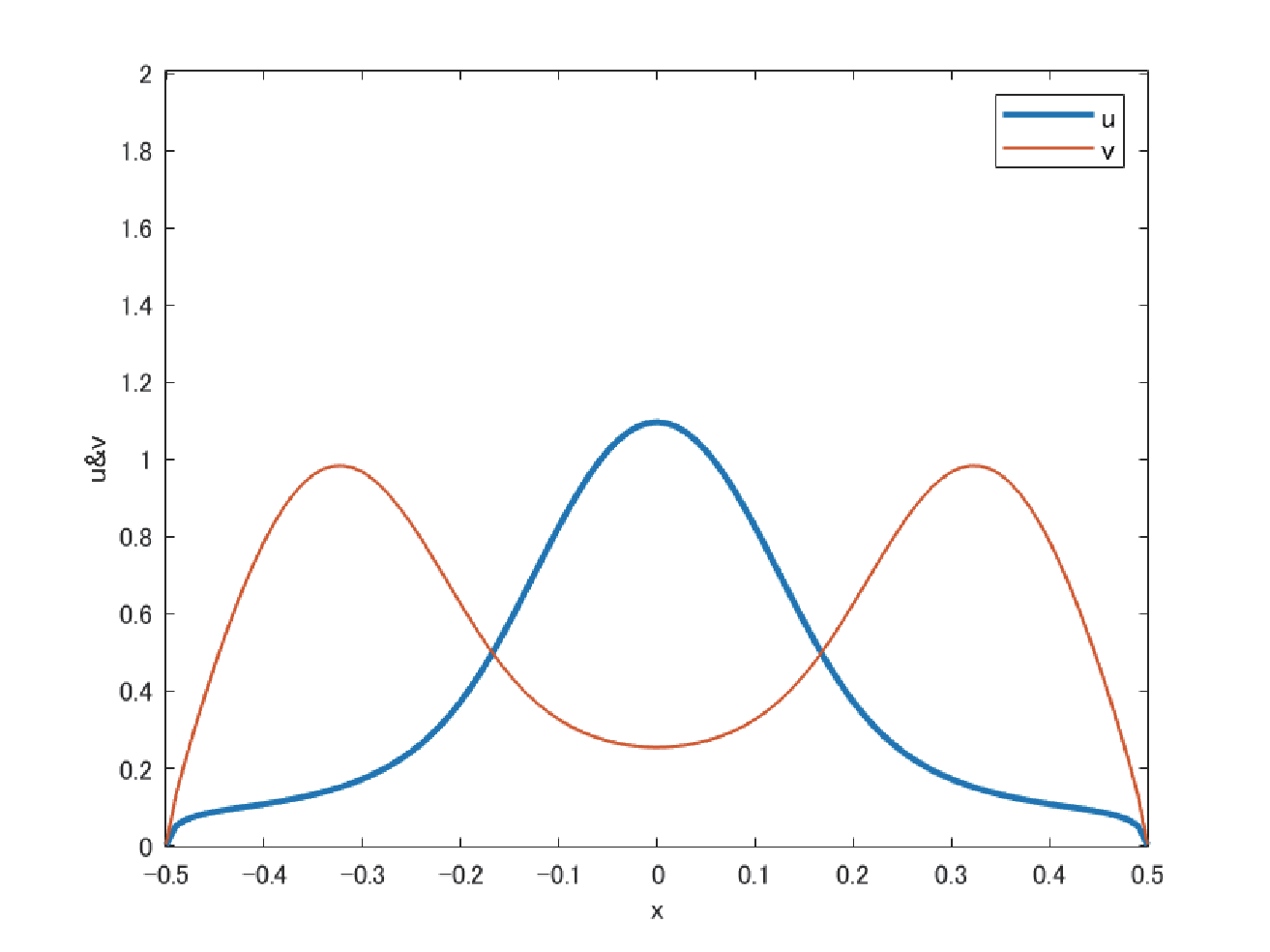}
\label{figc}}
\subfigure[Profile of a solution on purple lower branch in Figure 1 
at $\lambda=91.5836$.]{
\includegraphics*[scale=.3]{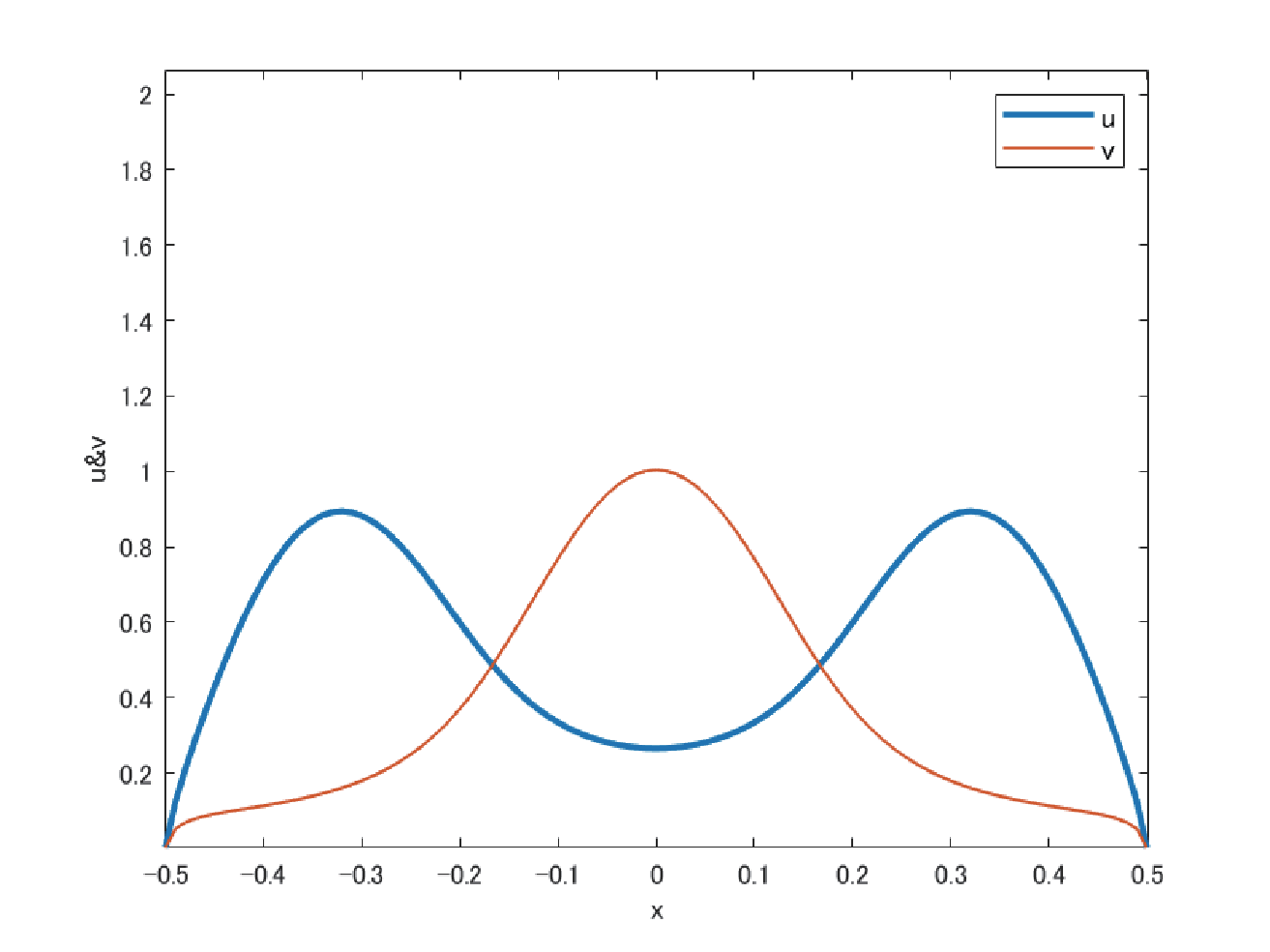}
\label{figd}
}
\caption{
Profiles of solutions on red and purple pitchfork bifurcation branches
in Figure 1.}
\label{fig3}
\end{figure}

\begin{figure}
\begin{center}
{\includegraphics*[scale=.4]{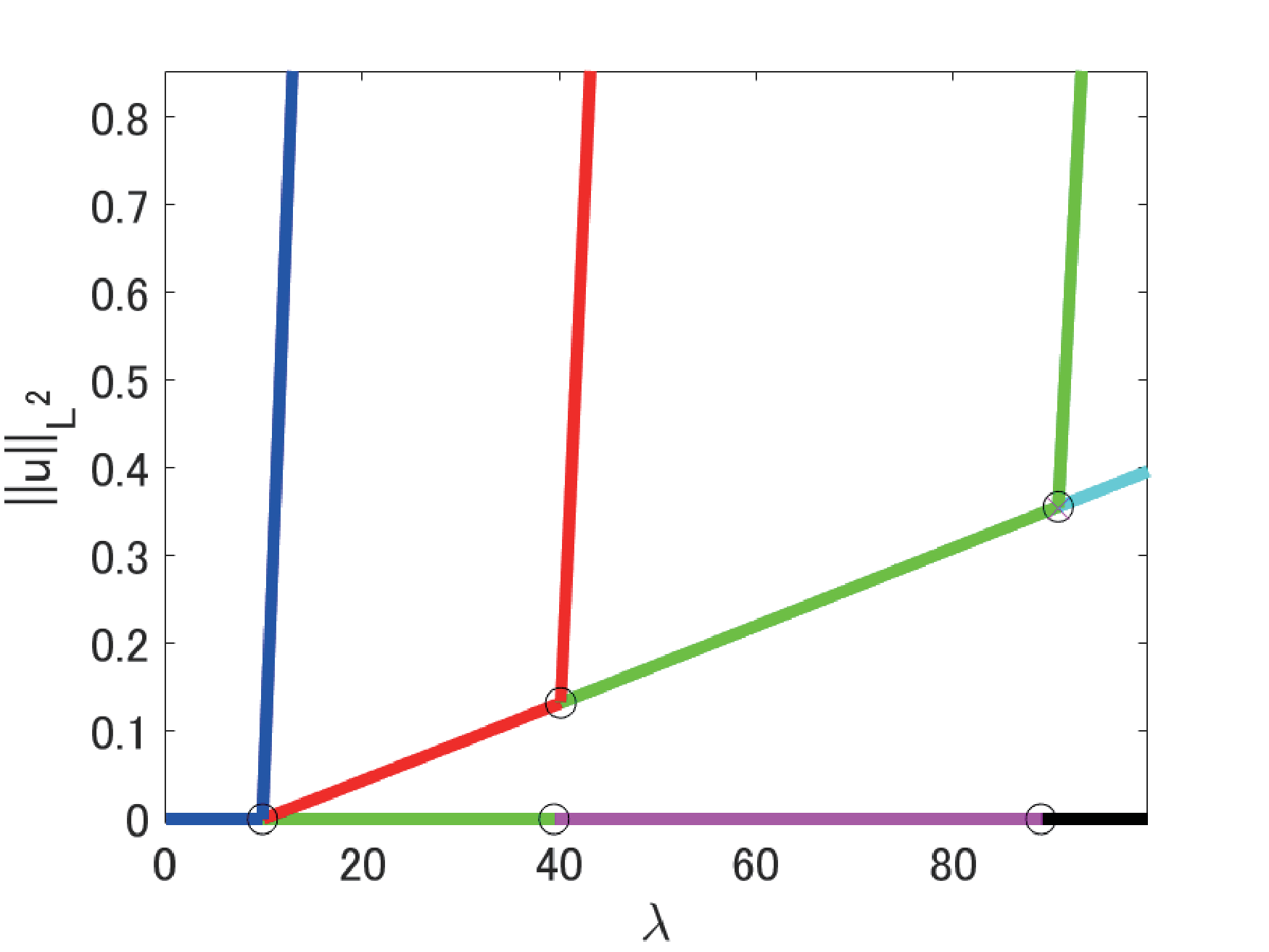}}\\
\caption{Morse index of solutions of \eqref{SP}.}
\end{center}\end{figure}

\section{Numerical results}
In this section, we exhibit the numerical results
on the Morse index of positive steady-states 
by using the continuation software \texttt{pde2path}
\cite{BKS, DRUW, Ue, UWR} based on an FEM discretization
of the stationary problem.
For \eqref{SP}, our setting of parameters in the numerical simulation 
is as follows:
\begin{equation}\label{setting}
\Omega=(-0.5,0.5),
\quad \alpha=20, 
\quad b_{1}=3, 
\quad b_{2}=2, 
\quad c_{1}=2, 
\quad c_{2}=1.
\end{equation}

Our previous paper with Inoue \cite{IKS}
has already shown the numerical bifurcation diagram as in Figure 1,
where
the horizontal axis represents the bifurcation parameter $\lambda$, 
and the vertical axis represents the $L^{2}$ norm of the 
$u$ component of positive solutions $(u,v)$ to \eqref{SP}.
In Figure 1,
the blue curve is corresponding to the branch 
$\mathcal{C}_{20, \varLambda}$ 
of small coexistence
(see \eqref{bif1})
bifurcating from the trivial solution at $\lambda=\lambda_{1}$.
The theoretical value of $\lambda_{1}$ with \eqref{setting} is
equal to
$\pi^2\,(\,\fallingdotseq 9.8696)$.
Figure 2, also a reprint from \cite{IKS}, 
shows the profile of a positive 
solution $(u,v)$ 
corresponding to the point at $\lambda=59.8286$
on the blue curve $\mathcal{C}_{20, \varLambda}$.
In Figure 1, 
the red curve exhibits
the upper branch $\mathcal{S}_{2,20,\varLambda}^{+}$
and lower one $\mathcal{S}_{2,20,\varLambda}^{-}$ 
(\eqref{seg+} and \eqref{seg-}).
The pitchfork bifurcation curve $\mathcal{S}_{2, 20, \varLambda}$
bifurcates from
a solution on the blue curve at $\beta_{2,20}$.
It should be noted the $L^{2}$ norm of each component
of $\mathcal{S}_{2,20, \varLambda}^{+}$ and 
$\mathcal{S}_{2,20, \varLambda}^{-}$
are shown overlapped.
It was shown in \cite{IKS} that 
the secondary bifurcation point 
$\beta_{2,\alpha}$ theoretically tends to 
$\lambda_{2}=(2\pi)^2\,(\,\fallingdotseq 39.4784)$
as $\alpha\to\infty$.
In Figure 3, also a reprint from \cite{IKS},
(a) and (b) exhibit the profiles of
solutions on the upper branch $\mathcal{S}^{+}_{2, 20, \varLambda}$ 
and the lower one $\mathcal{S}^{-}_{2, 20, \varLambda}$ 
of the red pitchfork bifurcation curve 
$\mathcal{S}_{2, 20, \varLambda}$
at
$\lambda=40.3421$, respectively.
It can be observed that $u$ and $v$ are somewhat spatially segregated.
Furthermore, as the solution moves away 
from the bifurcation point on the blue curve, 
the numerical simulation shows that 
the segregation between $u$ and $v$ becomes more overt. 
Actually, (c) and (d) in Figure 3 exhibit the profiles of
solutions on $\mathcal{S}^{+}_{2, 20, \varLambda}$ 
and $\mathcal{S}^{+}_{2, 20, \varLambda}$ 
on the red pitchfork bifurcation curve 
$\mathcal{S}_{2, 20, \varLambda}$ 
at
$\lambda=43.0673$, 
where $u$ and $v$ considerably segregate each other.

The numerical result of this paper is shown in Figure 4 exhibiting
the Morse index of each solutions of \eqref{SP} with \eqref{setting}.
It should be noted that the blue curve in Figure 4 bifurcating
from the trivial solution at $\lambda=\lambda_{1}
(=\pi^2\fallingdotseq 9.8696)$
is corresponding to the bifurcation curve of the semi-trivial
solutions $(u,v)=(\theta_{\lambda}, 0)$ for $\lambda>\lambda_{1}$,
which is not included in Figure 1 focusing only on positive solutions.
In Figure 4, the blue part is the stable branch of 
the trivial solution and the semi-trivial solutions;
the red part is the unstable branch of 
the solutions of small coexistence and of segregation
with the Morse index $1$;
the green part is the unstable branch of 
the solutions with the Morse index $2$,
the light blue part is unstable branch of 
the solutions with the Morse index $3$.

%
%

\end{document}